\documentclass[11pt]{article}

\usepackage[cp1252]{inputenc}
\usepackage[english]{babel}
\usepackage[a4paper]{geometry}
\usepackage{amsmath,amsfonts,amssymb,amsthm,cases,upgreek}
\usepackage{empheq}
\usepackage{stmaryrd}

\usepackage{diagbox}

\usepackage{dsfont}
\usepackage{graphicx}
\usepackage{comment}
\usepackage{bm} 

\newtheorem{thm}{Theorem}[section]
\newtheorem{definition}[thm]{Definition}
\newtheorem{lemma}[thm]{Lemma}

\newtheorem{prop}[thm]{Proposition}
\newtheorem{remark}[thm]{Remark}

\newcommand{\RR}{\mathbb{R}^{2}}

\newcommand{\RRD}{\mathbb{R}^{3}}

\newcommand{\Umus}{\textbf{U}^{\mu,\gamma}_{\! \sslash}}

\newcommand{\lver}{\left\lvert}
\newcommand{\llver}{\left\lvert\left\lvert}
\newcommand{\rver}{\right\rvert}
\newcommand{\rrver}{\right\rvert \right\rvert}

\newcommand{\Vastsh}{\textbf{V}^{\ast}_{\text{sh}}}

\renewcommand{\footnote}[1]{\textsuperscript{\addtocounter{footnote}{1}(\thefootnote)}\footnotetext{#1}}

\def \epsilon {\varepsilon}

\SetSymbolFont{stmry}{bold}{U}{stmry}{m}{n}

\begin{document}
\title{\textbf{The KP approximation under a weak Coriolis forcing}}
\author{Benjamin MELINAND\footnote{Indiana University. Email : bmelinan@indiana.edu}}
\date{October 2017}

\maketitle

\vspace{-0.3cm}

\begin{abstract}
\noindent In this paper, we study the asymptotic behavior of weakly transverse water-waves under a weak Coriolis forcing in the long wave regime. We derive the Boussinesq-Coriolis equations in this setting and we provide a rigorous justification of this model. Then, from these equations, we derive two other asymptotic models. When the Coriolis forcing is weak, we fully justify the rotation-modified Kadomtsev-Petviashvili equation (also called Grimshaw-Melville equation). When the Coriolis forcing is very weak, we rigorously justify the Kadomtsev-Petviashvili equation. This work provides the first mathematical justification of the KP approximation under a Coriolis forcing.
\end{abstract}

\section{Introduction}\label{Intro}

\noindent We consider the motion of an inviscid, incompressible fluid under the influence of the gravity $\bm{g} = -g \bm{e_{z}}$ and the rotation of the Earth with a rotation vector $\textbf{f} = \frac{f}{2} \bm{e_{z}}$. We assume that the fluid has a constant density $\rho$ and that no surface tension is involved.  We assume that the surface is a graph above the still water level and that the seabed is flat. We denote by $X=(x,y) \in \RR$ the horizontal variable and by $z \in \mathbb{R}$ the vertical variable. The fluid occupies the domain $\Omega_{t} := \{ (X,z) \in \RRD \text{ , } -H < z < \zeta (t,X) \}$. We denote by $\textbf{U} = \left( \textbf{V}, \text{w} \right)^{t}$ the velocity in the fluid. Notice that $\textbf{V}$ is the horizontal component of $\textbf{U}$ and $\text{w}$ its vertical component. Finally, we assume that the pressure $\mathcal{P}$ is constant at the surface of the fluid. The  equations governing such a fluid are the free surface Euler-Coriolis equations\footnote{The centrifugal potential is assumed to be constant and included in the pressure term.}

\begin{equation}\label{Euler_equations}
\left\{
\begin{aligned}
&\partial_{t} \textbf{U} + \left( \textbf{U} \cdot \nabla_{\! X,z} \right) \textbf{U} + \textbf{f} \times \textbf{U} = - \frac{1}{\rho} \nabla_{\! X,z} \mathcal{P} - g \bm{e_{z}} \text{ in } \Omega_{t},\\
&\text{div} \; \textbf{U} = 0 \text{ in } \Omega_{t},
\end{aligned}
\right.
\end{equation}

\noindent with the boundary conditions

\begin{equation*}
\left\{
\begin{aligned}
&\mathcal{P}_{|z=\zeta} = P_{0}, \\
&\partial_{t} \zeta	- \underline{\textbf{U}} \cdot \textbf{N} = 0, \\
&\text{w}_{b} = 0,
\end{aligned}
\right.
\end{equation*}

\noindent where $P_{0}$ is constant, $\textbf{N} = \begin{pmatrix} - \nabla \zeta \\ 1 \end{pmatrix}$, $\underline{\textbf{U}} = \begin{pmatrix} \underline{\textbf{V}} \\ \underline{\text{w}} \end{pmatrix} = \textbf{U}_{|z=\zeta}$ and $\textbf{U}_{b} = \begin{pmatrix} \textbf{V}_{b} \\ \text{w}_{b} \end{pmatrix} = \textbf{U}_{|z=-H}.$

\medskip

\noindent In this work, we do not directly work on the free surface Euler-Coriolis equations. We rather consider another formulation called the Castro-Lannes formulation (see \cite{Castro_Lannes_vorticity}). This formulation generalizes the well-known Zakharov/Craig-Sulem formulation (\cite{Zakharov,Craig_Sulem_2}) to a fluid with a rotational component. In \cite{Castro_Lannes_vorticity}, Castro and Lannes  shown that we can express the free surface Euler equations thanks to the unknowns $\left(\zeta, \textbf{U}_{\! \sslash}, \bm{\omega} \right)$\footnote{Notice that Castro and Lannes used the unknowns $\left(\zeta, \frac{\nabla}{\Delta} \cdot \textbf{U}_{\! \sslash}, \bm{\omega} \right)$. However, as noticed in \cite{my_long_wave_corio}, the unknowns $\left(\zeta, \textbf{U}_{\! \sslash}, \bm{\omega} \right)$ are better to derive shallow water asymptotic models.} where $\bm{\omega} = \text{Curl} \; \textbf{U}$ is the vorticity of the fluid and

\begin{equation*}
\textbf{U}_{\! \sslash} = \underline{\textbf{V}} + \underline{\text{w}} \nabla \zeta.
\end{equation*}

\noindent Then, they provide a system of three equations on these unknowns. In \cite{my_coriolis}, a similar work has been done to take into account the Coriolis forcing. It leads to the following system, called the Castro-Lannes system or the water waves equations with vorticity,

\begin{equation}\label{castro_lannes_formulation_dim}
\left\{
\begin{aligned}
&\hspace{-0.05cm} \partial_{t} \zeta -\underline{\textbf{U}}\cdot \textbf{N} = 0,\\
&\partial_{t} \textbf{U}_{\! \sslash} \hspace{-0.1cm} +  \hspace{-0.1cm} \nabla \zeta  \hspace{-0.1cm} +  \hspace{-0.1cm} \frac{1}{2} \nabla \hspace{-0.1cm} \left\lvert \textbf{U}_{\! \sslash} \right\rvert^{2}  \hspace{-0.1cm} -  \hspace{-0.1cm} \frac{1}{2} \nabla \hspace{-0.1cm} \left[ \left( 1 +  \left\lvert \nabla \zeta\right\rvert^{2} \right) \underline{\text{w}}^{2} \right]  \hspace{-0.1cm} +  \hspace{-0.1cm} \left( \nabla^{\perp} \cdot \textbf{U}_{\! \sslash} \right) \underline{\textbf{V}}^{\perp}  \hspace{-0.1cm} +  \hspace{-0.1cm} f \underline{\textbf{V}}^{\perp} = 0,\\
&\hspace{-0.05cm} \partial_{t} \bm{\omega} \hspace{-0.05cm} + \hspace{-0.05cm}  \left( \textbf{U} \hspace{-0.05cm} \cdot \hspace{-0.05cm} \nabla_{\! X,z} \hspace{-0.05cm} \right) \bm{\omega} \hspace{-0.05cm} = \left( \bm{\omega} \cdot \nabla_{\! X,z} \right) \textbf{U} \hspace{-0.05cm} + \hspace{-0.05cm} f \partial_{z} \textbf{U},
\end{aligned}
\right.
\end{equation}

\noindent where $\textbf{U} = \begin{pmatrix} \textbf{V} \\ \text{w} \end{pmatrix} = \textbf{U}[\zeta](\textbf{U}_{\! \sslash},\bm{\omega})$ is the unique solution in $H^{1}(\Omega_{t})$

\noindent of the following Div-Curl equation

\begin{equation*}
\left\{
\begin{aligned}
&\text{curl} \; \textbf{U} = \bm{\omega} \text{ in } \Omega_{t},\\
&\text{div} \; \textbf{U} = 0 \text{ in } \Omega_{t},\\
&\left(\underline{\textbf{V}} + \underline{\text{w}} \nabla \zeta \right)_{|z= \zeta} = \textbf{U}_{\! \sslash},\\
&\text{w}_{b} = 0.
\end{aligned}
\right.
\end{equation*}

\noindent The main goal of this paper is to study weakly transverse long waves. Therefore, we consider a nondimensionalization of the previous equations. Five physical parameters are involved in this work : the typical amplitude of the surface $a$, the typical longitudinal scale $L_{x}$, the typical transverse scale $L_{y}$, the characteristic water depth $H$ and the typical Coriolis frequency $f$. We introduce four dimensionless parameters 

\begin{equation*}
\mu = \frac{H^{2}}{L_{x}^{2}} \text{, } \epsilon =\frac{a}{H} \text{, } \text{Ro} = \frac{a \sqrt{gH}}{H f L_{x}} \text{ and } \gamma = \frac{L_{x}}{L_{y}}.
\end{equation*}

\noindent The parameter $\mu$ is called the shallowness parameter. The parameter $\epsilon$ is called the nonlinearity parameter. The parameter $\text{Ro}$ is the Rossby number and finally the parameter $\gamma$ is called the transversality parameter. Then, we can nondimensionalize the Euler equations \eqref{Euler_equations} and the Castro-Lannes equations \eqref{castro_lannes_formulation_dim} (see Part \ref{nondim}). In this work, we study the following asymptotic regime 

\begin{equation*}
\mathcal{A}_{\text{boussi}} = \left\{ \left(\mu, \epsilon, \gamma, \text{Ro}  \right), 0 \leq \mu \leq \mu_{0}, \epsilon = \mathcal{O} \left( \mu \right), \gamma \leq 1, \frac{\epsilon}{\text{Ro}} = \mathcal{O}(\sqrt{\mu}) \right\},
\end{equation*}

\noindent This regime corresponds to a long wave regime ($\epsilon = \mathcal{O}(\mu)$) under a weak Coriolis forcing $\frac{\epsilon}{\text{Ro}} = \mathcal{O}(\sqrt{\mu})$. For an explanation of the first assumption, we refer to \cite{Lannes_ww}. The second assumption is standard in oceanography. Rewriting $\frac{\epsilon}{\text{Ro}}= \frac{f L_{x}}{\sqrt{gH}}$, this assumption means that the rotation period is assumed to be much smaller than the time scale of the waves. We refer to \cite{grimshaw_intern,germain-renouard} for more explanations about this assumption (see also \cite{grimshaw_KP,pedlosky,gill,leblond}).

\medskip

\noindent We organize this paper in four parts. In Section \ref{nondim}, we explain how we nondimensionalize the equations and we provide a local wellposedness result. In Section \ref{boussi_deriv}, we derive and justify the Boussinesq-Coriolis equations in the asymptotic regime $\mathcal{A}_{\text{boussi}}$. The Boussinesq-Coriolis equations are a system of three equations on the surface $\zeta$ and the vertical average of the horizontal velocity denoted $\overline{\textbf{V}}$ (defined in \eqref{V_average}). They  correspond to a $\mathcal{O}(\mu^{2})$ approximation of the water waves equations. These equations are

\begin{equation}\label{boussinesq_eq}
\left\{
\begin{aligned}
&\partial_{t} \zeta + \nabla^{\gamma} \cdot \left( \left[1 + \epsilon \zeta \right] \overline{\textbf{V}} \right) = 0,\\
&\left(1- \frac{\mu}{3} \nabla^{\gamma} \nabla^{\gamma} \cdot \right) \partial_{t} \overline{\textbf{V}} + \nabla^{\gamma} \zeta + \epsilon \overline{\textbf{V}} \cdot \nabla^{\gamma} \overline{\textbf{V}} + \frac{\epsilon}{\text{Ro}} \overline{\textbf{V}}^{\perp} = 0.
\end{aligned}
\right.
\end{equation}

\noindent Then, in Section \ref{KP_approx}, we study the KP approximation which corresponds to the asymptotic regime $\mathcal{A}_{\text{boussi}}$ with $\epsilon = \mu$ and $\gamma = \sqrt{\mu}$. This second assumption corresponds to weakly transverse effects (see for instance \cite{Lannes_ww}). In this regime, we derive two other asymptotic models.  When the Coriolis forcing is weak $\left(\frac{\epsilon}{\text{Ro}} = \sqrt{\mu} \right)$, we rigorously justify the modified-rotation Kadomtsev-Petviashvili equation (see Subsection \ref{model_RKP}), also called Grimshaw-Melville equation in the physics literature, 

\begin{equation*}
\partial_{\xi} \left( \partial_{\tau} k +  \frac{3}{2} k \partial_{\xi} k + \frac{1}{6} \partial_{\xi}^{3} k \right) + \frac{1}{2} \partial_{yy} k = \frac{1}{2} k.
\end{equation*}
 
\noindent Then, when the Coriolis forcing is very weak $\left(\frac{\epsilon}{\text{Ro}} = \mu \right)$, we fully justify the KP equation (see Subsection \ref{model_KP})

\begin{equation*}
\partial_{\xi} \left( \partial_{\tau} k +  \frac{3}{2} k \partial_{\xi} k + \frac{1}{6} \partial_{\xi}^{3} k \right) + \frac{1}{2} \partial_{yy} k = 0.
\end{equation*}

\noindent Finally, in Section \ref{comparaison_models}, we compare the scalar asymptotic models we derive in Section \ref{KP_approx} with the ones derived in \cite{my_long_wave_corio} : the Ostrovsky equation and the KdV equation.

\subsection{Notations/Definitions}

\noindent - If $\textbf{A} \in \RRD$, we denote by $\textbf{A}_{h}$ its horizontal component.

\medskip

\noindent - If $\textbf{V}  = \begin{pmatrix} u \\ v \end{pmatrix} \in \RR$, we define the orthogonal of $\textbf{V}$ by $\textbf{V}^{\perp} = \begin{pmatrix} -v \\ u \end{pmatrix}$.

\medskip

\noindent - In this paper, $C \left( \cdot \right)$ is a nondecreasing and positive function whose exact value has no importance.

\medskip

\noindent - Consider a vector field $\textbf{A}$ or a function $\text{w}$ defined on $\Omega$. Then, we denote $\underline{\textbf{A}} = \textbf{A}_{|z=\epsilon\zeta}$, $\underline{\text{w}} = \text{w}_{|z=\epsilon \zeta}$ and $\textbf{A}_{b} = \textbf{A}_{|z=-1}$, $\text{w}_{b} = \text{w}_{|z=-1}$.

\medskip

\noindent - If $N \in \mathbb{N}$ and $f$ is a function on $\RR$, $\left\lvert f \right\rvert_{H^{N}}$ is its $H^{N}$-norm, $\left\lvert f \right\rvert_{2}$ is its $L^{2}$-norm and $\lver f \rver_{L^{\infty}}$ its  $L^{\infty}$-norm. We denote by $\left(\; , \; \right)_{2}$ the $L^{2}(\RR)$ inner product.

\medskip

\noindent - If $f$ is a function defined on $\RR$, We use the notation $\nabla^{\gamma} f = \left( \partial_{x} f, \gamma \partial_{y} f\right)^{t}$.

\medskip

\noindent - If $u = u(X,z)$ is defined in $\Omega$, we define

\begin{equation*}
\overline{u}(X) = \frac{1}{1+\epsilon \zeta} \int_{-1}^{\epsilon \zeta(X)} u(X,z) dz \text{   and   } u^{\ast} = u - \overline{u}. 
\end{equation*}

\medskip

\noindent - For $N  \geq 0$, we define the Hilbert spaces $\partial_{x} H^{N}(\RR)$ 

\begin{equation}\label{partialx_HN}
\partial_{x} H^{N}(\RR) =  \left\{ k \in H^{N-1}(\RR) \text{,  } k = \partial_{x} \tilde{k} \text{  with  } \tilde{k} \in H^{N}(\RR)\right\}.
\end{equation}

\noindent The function $\tilde{k}$  is denoted $\partial_{x}^{-1} k$ in the following.

\medskip

\noindent - Similarly, for $N  \geq 0$, we define the Hilbert spaces $\partial^{2}_{x} H^{N}(\RR)$.

\medskip

\noindent - In the following definition, we recall the notion of consistence (see for instance \cite{Lannes_ww}). 

\begin{definition}\label{consistence}
\noindent We say that the Castro-Lannes equations \eqref{castro_lannes_formulation} are consistent of order $\mathcal{O} \left( \mu^{2} \right)$ with a system of equations $S$ for $\zeta$ and \upshape $\overline{\textbf{V}}$ \itshape if for any smooth solutions \upshape $\left(\zeta, \Umus, \bm{\omega} \right)$ \itshape of the Castro-Lannes equations \eqref{castro_lannes_formulation}, the pair \upshape $\left(\zeta, \overline{\textbf{V}}[\epsilon \zeta] \left(\Umus, \bm{\omega} \right) \right)$ (defined in \eqref{V_average}) \itshape solves $S$ up to a residual of order $\mathcal{O} \left( \mu^{2} \right)$. 
\end{definition}

\subsection{Nondimensionalization}\label{nondim}

\noindent We recall the four dimensionless parameters

\begin{equation}\label{parameters}
\mu = \frac{H^{2}}{L_{x}^{2}} \text{, } \epsilon =\frac{a}{H} \text{, } \text{Ro} = \frac{a \sqrt{gH}}{H f L_{x}} \text{ and } \gamma = \frac{L_{x}}{L_{y}}.
\end{equation}

\noindent  We nondimensionalize the variables and the unknowns. We introduce (see \cite{Lannes_ww} or \cite{my_coriolis})

\begin{equation*}
 \left\{ 
 \begin{aligned}
 &x' = \frac{x}{L_{x}} \text{, } y' = \frac{y}{L_{y}} \text{, } z' = \frac{z}{H} \text{, } \zeta' = \frac{\zeta}{a} \text{, } t' = \frac{\sqrt{gH}}{L_{x}} t \text{, }\\
 &\textbf{V}' = \sqrt{\frac{H}{g}} \frac{\textbf{V}}{a} \text{, } \text{w}'= H \sqrt{\frac{H}{g}} \frac{\text{w}}{aL_{x}} \text{ and } \mathcal{P}' = \frac{\mathcal{P}}{\rho g H}.
 \end{aligned}
 \right.
\end{equation*}

\noindent In the following, we use the following notations

\begin{equation*}
\nabla^{\gamma} = \nabla^{\gamma}_{\! X'} = \begin{pmatrix} \partial_{x'} \\ \gamma \partial_{y'} \end{pmatrix} \text{  ,  } \nabla^{\mu, \gamma}_{\! X',z'} = \begin{pmatrix} \sqrt{\mu} \nabla^{\gamma}_{\! X'} \\ \partial_{z'} \end{pmatrix} \text{  ,  } \text{curl}^{\mu, \gamma} = \nabla^{\mu, \gamma}_{\! X',z'} \times \text{, } \text{div}^{\mu, \gamma} = \nabla^{\mu, \gamma}_{\! X',z'} \cdot .
\end{equation*}

\noindent We also define 

\begin{equation}\label{nondim_U}
\textbf{U}^{\mu} = \begin{pmatrix} \sqrt{\mu} \textbf{V}' \\ \text{w}' \end{pmatrix} \text{, } \bm{\omega}'=\frac{1}{\mu} \text{curl}^{\mu, \gamma} \textbf{U}^{\mu},
\end{equation}

\noindent and 

\begin{equation*}
\underline{\textbf{U}}^{\mu} = \begin{pmatrix} \sqrt{\mu} \underline{\textbf{V}}' \\ \underline{\text{w}}' \end{pmatrix} = \textbf{U}^{\mu}_{|z'= \epsilon \zeta'} \text{, } \textbf{U}^{\mu}_{b} = \textbf{U}^{\mu}_{|z'=-1}, \textbf{N}^{\mu, \gamma} =  \begin{pmatrix} - \epsilon \sqrt{\mu} \nabla^{\gamma} \zeta' \\ 1 \end{pmatrix}.
\end{equation*}

\begin{remark}
\noindent Notice that the nondimensionalization of the vorticity presented in \eqref{nondim_U} corresponds to weakly sheared flows (see \cite{Castro_Lannes_shallow_water}, \cite{teshukov_shear_flows}, \cite{richard_gravi_shear_flows}).
\end{remark}

\noindent The nondimensionalized fluid domain is

\begin{equation*}
\Omega^{\prime}_{t'} := \{ (X',z') \in \RRD \text{ , } -1 < z' < \epsilon \zeta' (t',X') \}.
\end{equation*}

\noindent Finally, the Euler-Coriolis equations \eqref{Euler_equations} become 

\begin{equation*}
\left\{
\begin{aligned}
&\partial_{t'} \textbf{U}^{\mu} + \frac{\epsilon}{\mu} \left( \textbf{U}^{\mu} \cdot \nabla_{\! X',z'}^{\mu, \gamma} \right) \textbf{U}^{\mu} + \frac{\epsilon \sqrt{\mu}}{\text{Ro}} \begin{pmatrix} \;\; \textbf{V}^{\prime \perp} \\ 0 \end{pmatrix} = - \frac{1}{\epsilon} \nabla_{\! X',z'}^{\mu, \gamma} \mathcal{P}' - \frac{1}{\epsilon} \bm{e_{z}} \text{ in } \Omega^{\prime}_{t},\\
&\text{div}^{\mu, \gamma}_{\! X' \! ,z'} \; \textbf{U}^{\mu} = 0 \text{ in } \Omega^{\prime}_{t},
\end{aligned}
\right.
\end{equation*}

\noindent with the boundary conditions

\begin{equation*}
\left\{
\begin{aligned}
&\partial_{t'} \zeta'	- \frac{1}{\mu} \underline{\textbf{U}}^{\mu} \cdot \textbf{N}^{\mu, \gamma} = 0, \\
&\text{w}'_{b} = 0.
\end{aligned}
\right.
\end{equation*}

\noindent In the following, we omit the primes. We can proceed similarly to nondimensionalize the Castro-Lannes formulation. We define the quantity

\begin{equation*}
\Umus = \underline{\textbf{V}} + \epsilon \underline{\text{w}} \nabla^{\gamma} \zeta.
\end{equation*}

\noindent Then, the Castro-Lannes formulation becomes (see \cite{Castro_Lannes_vorticity} or \cite{my_coriolis} when $\gamma = 1$),

\small
\begin{equation}\label{castro_lannes_formulation}
\left\{
\begin{aligned}
&\hspace{-0.05cm} \partial_{t} \zeta - \frac{1}{\mu} \underline{\textbf{U}}^{\mu} \cdot \textbf{N}^{\mu, \gamma} = 0,\\
&\partial_{t} \Umus \hspace{-0.1cm} +  \hspace{-0.1cm} \nabla^{\gamma} \zeta  \hspace{-0.1cm} +  \hspace{-0.1cm} \frac{\epsilon}{2} \nabla^{\gamma}  \hspace{-0.1cm} \left\lvert \Umus \right\rvert^{2}  \hspace{-0.1cm} -  \hspace{-0.1cm} \frac{\epsilon}{2 \mu} \nabla^{\gamma}  \hspace{-0.1cm} \left[ \left( 1 +  \epsilon^{2} \mu \left\lvert \nabla^{\gamma} \zeta\right\rvert^{2} \right) \underline{\text{w}}^{2} \right]  \hspace{-0.1cm} +  \hspace{-0.1cm} \epsilon \hspace{-0.05cm} \left(\hspace{-0.05cm} \nabla^{\perp} \hspace{-0.05cm} \cdot \hspace{-0.05cm} \Umus \hspace{-0.05cm} \right) \hspace{0.05cm} \underline{\textbf{V}}^{\perp}  \hspace{-0.1cm} +  \hspace{-0.1cm} \frac{\epsilon}{\text{Ro}} \underline{\textbf{V}}^{\perp} = 0,\\
&\hspace{-0.05cm} \partial_{t} \bm{\omega} \hspace{-0.05cm} + \hspace{-0.05cm} \frac{\epsilon}{\mu} \hspace{-0.05cm} \left( \hspace{-0.05cm} \textbf{U}^{\mu} \hspace{-0.05cm} \cdot \hspace{-0.05cm} \nabla^{\mu, \gamma}_{\! X,z} \hspace{-0.05cm} \right) \bm{\omega} \hspace{-0.05cm} = \hspace{-0.05cm} \frac{\epsilon}{\mu} \left( \bm{\omega} \cdot \nabla^{\mu, \gamma}_{\! X,z} \right) \textbf{U}^{\mu} \hspace{-0.05cm} + \hspace{-0.05cm} \frac{\epsilon}{\mu \text{Ro}} \partial_{z} \textbf{U}^{\mu},
\end{aligned}
\right.
\end{equation}
\normalsize

\noindent where $\textbf{U}^{\mu} = \begin{pmatrix} \sqrt{\mu} \textbf{V} \\ \text{w} \end{pmatrix} = \textbf{U}^{\mu}[\epsilon \zeta](\Umus,\bm{\omega})$ is the unique solution in $H^{1}(\Omega_{t})$ of

\begin{equation*}
\left\{
\begin{aligned}
&\text{curl}^{\mu, \gamma} \; \textbf{U}^{\mu} = \mu \bm{\omega} \text{ in } \Omega_{t},\\
&\text{div}^{\mu, \gamma} \; \textbf{U}^{\mu} = 0 \text{ in } \Omega_{t},\\
&\left(\underline{\textbf{V}} + \epsilon \underline{\text{w}} \nabla^{\gamma} \zeta \right)_{|z=\epsilon \zeta} = \Umus,\\
&\text{w}_{b} = 0.
\end{aligned}
\right.
\end{equation*}

\noindent In order to rigorously derive asymptotic models, we need an existence result for the Castro-Lannes formulation \eqref{castro_lannes_formulation}. We recall that the existence of solutions to the water waves equations is always obtained under the so-called Rayleigh-Taylor condition that assumes the positivity of the Rayleigh-Taylor coefficient $\mathfrak{a}$ (see Part 3.4.5 in \cite{Lannes_ww} for the link between $\mathfrak{a}$ and the Rayleigh-Taylor condition or \cite{my_coriolis})  where

\begin{equation*}
\mathfrak{a}: = \mathfrak{a}[\epsilon \zeta](\Umus, \bm{\omega}) = 1 + \epsilon \left( \partial_{t} + \epsilon \underline{\textbf{V}}[\epsilon \zeta](\Umus, \bm{\omega}) \cdot \nabla \right) \underline{\text{w}}[\epsilon \zeta](\Umus, \bm{\omega}).
\end{equation*}

\noindent We explain in \cite{my_coriolis} how we can define the Rayleigh-Taylor coefficient $\mathfrak{a}$ at $t=0$. We also assume that the water depth is bounded from below by a positive constant

\begin{equation*}
\exists \, h_{\min} > 0 \text{ ,  } 1 + \epsilon \zeta \geq h_{\min}.
\end{equation*}

\noindent The following theorem can be found in \cite{my_coriolis} and provide a local wellposedness result of the Castro-Lannes formulation \eqref{castro_lannes_formulation} (see also Theorem 1.5 in \cite{my_long_wave_corio}).
 
\begin{thm}\label{existence_ww}
\noindent Let \upshape$A > 0$ and $\textbf{N} \geq 5$. We suppose that \upshape$\left( \mu,\epsilon,\gamma,\text{Ro} \right) \in \mathcal{A}_{\text{boussi}}$\itshape. We assume that

\upshape
\begin{equation*}
\left( \zeta_{0}, (\Umus)_{0}, \bm{\omega}_{0} \right) \in H^{{N+\frac{1}{2}}}(\RR) \times H^{N}(\RR) \times H^{N-1}(\Omega_{0}),
\end{equation*}
\itshape

\noindent with \upshape $\nabla^{\mu, \gamma} \cdot \bm{\omega}_{0} =0$ \itshape and \upshape $\nabla^{\gamma \perp} \cdot (\Umus)_{0} = \underline{\bm{\omega}_{0}} \cdot \begin{pmatrix} - \epsilon \sqrt{\mu} \nabla^{\gamma} \zeta_{0} \\ 1 \end{pmatrix}$ \itshape. Finally, we assume that

\upshape
\begin{equation*}
\exists \, h_{\min} \text{, } \mathfrak{a}_{\min} > 0 \text{ ,  } 1 + \epsilon \zeta _{0}  \geq h_{\min} \text{  and  } \mathfrak{a}[\epsilon \zeta_{0}]((\Umus)_{0}, \bm{\omega}_{0}) \geq \mathfrak{a}_{\min},
\end{equation*}
\itshape

\noindent and that

\upshape
\begin{equation*}
\lver \zeta_{0} \rver_{H^{N+\frac{1}{2}}} + \lver \frac{1}{\sqrt{1+ \sqrt{\mu} |D|}} (\Umus)_{0} \rver_{H^{N}} + \llver \bm{\omega}_{0} \rrver_{H^{N-1}} \leq A.
\end{equation*}
\itshape
 
\noindent Then, there exists \upshape$T > 0$ \itshape and a unique classical solution \upshape$\left( \zeta, \Umus, \bm{\omega} \right) $ \itshape to the Castro-Lannes \eqref{castro_lannes_formulation} on $[0,T]$ with initial data \upshape$\left( \zeta_{0}, (\Umus)_{0}, \bm{\omega}_{0} \right)$\itshape. Moreover, 

\upshape
\begin{align*}
&T =  \frac{T_{0}}{\max \left(\epsilon, \frac{\epsilon}{\text{Ro}} \right)} \text{ , }  \frac{1}{T_{0}} = c^{1},\\
&\underset{[0,T]}{\max} \left( \lver \zeta(t, \cdot) \rver_{H^{N}} + \lver \frac{1}{\sqrt{1+ \sqrt{\mu} |D|}} \Umus(t, \cdot) \rver_{H^{N-\frac{1}{2}}} + \llver \bm{\omega}(t, \cdot) \rrver_{H^{N-1}}  \right) = c^{2},
\end{align*}
\itshape

\noindent with \upshape $c^{j} = C \left(A, \mu_{0}, \frac{1}{h_{\min}}, \frac{1}{\mathfrak{a}_{\min}} \right)$.\itshape
\end{thm}

\begin{remark}\label{control_remainder}
\noindent Notice that thanks to Theorem \ref{existence_ww} together with Part 5.5.1 in \cite{Castro_Lannes_vorticity}, the quantities \upshape $\zeta$, $\Umus$, $\bm{\omega}$ , $\overline{\textbf{V}}$, $\textbf{U}$, $\partial_{t} \zeta$, $\partial_{t} \Umus$, $\partial_{t} \bm{\omega}$ and $\partial_{t} \textbf{U}$ \itshape remain bounded uniformly with respect to the small parameters during the time evolution of the flow.
\end{remark}

\section{The Boussinesq-Coriolis equations}\label{boussi_deriv}

\noindent In this part, we derive and fully justify the Boussinesq-Coriolis equations \eqref{boussinesq_eq} under a weak Coriolis forcing $\frac{\epsilon}{\text{Ro}} = \mathcal{O} \left( \sqrt{\mu} \right)$. We recall the corresponding asymptotic regime

\begin{equation}\label{boussi_regime}
\mathcal{A}_{\text{boussi}} = \left\{ \left(\mu, \epsilon, \gamma, \text{Ro}  \right), 0 \leq \mu \leq \mu_{0}, \epsilon =  \mathcal{O}(\mu), \gamma \leq 1, \frac{\epsilon}{\text{Ro}} = \mathcal{O}(\sqrt{\mu}) \right\}.
\end{equation}

\noindent Notice that no assumption on $\gamma$ is made in this part. The Boussinesq equations correspond to an order $\mathcal{O}(\mu^{2})$ approximation of the water waves equations. Motivated by \cite{my_long_wave_corio},  we use the Castro-Lannes equations \eqref{castro_lannes_formulation} to derive this asymptotic model. We introduce the water depth

\begin{equation*}
h(t,X) = 1 + \epsilon \zeta(t,X),
\end{equation*}

\noindent and the vertical average of the horizontal velocity

\begin{equation}\label{V_average}
\overline{\textbf{V}} = \overline{\textbf{V}}[\epsilon \zeta](\Umus,\bm{\omega})(t,X) = \frac{1}{h(t,X)} \int_{z=-1}^{\epsilon \zeta(t,X)} \textbf{V}[\epsilon \zeta, \beta b](\Umus,\bm{\omega})(t,X,z) dz.
\end{equation}

\noindent In the following we denote $\textbf{V} = \left(u, v \right)^{t}$. More generally, if $u$ is a function defined in $\Omega$, we denote by $\overline{u}$ its vertical average and $u^{\ast} = u - \overline{u}$. We also have to introduce the "shear" velocity

\begin{equation*}
\textbf{V}_{\text{sh}} = \textbf{V}_{\text{sh}}[\epsilon \zeta](\Umus,\bm{\omega})(t,X)= \int_{z}^{\epsilon \zeta} \bm{\omega}_{h}^{\perp} (t,X,z') dz'
\end{equation*}

\noindent and its vertical average

\begin{equation*}
\textbf{Q} = \overline{\textbf{V}_{\text{sh}}} = \frac{1}{h} \int_{-1}^{\epsilon \zeta} \int_{z'}^{\epsilon \zeta} \bm{\omega}_{h}^{\perp}.
\end{equation*}

\noindent As noticed in \cite{Castro_Lannes_vorticity}, these quantities appear when one wants to obtain an expansion with respect to $\mu$ of the velocity. We recall that

\begin{equation*}
\Umus = \underline{\textbf{V}} + \epsilon \underline{\text{w}} \nabla^{\gamma} \zeta.
\end{equation*}

\subsection{Asymptotic expansions with respect to $\mu$}

\noindent In this part, we give an expansion of different quantities with respect to $\mu$. These expansions will help us to derive the Boussinesq-Coriolis equations \eqref{boussinesq_eq} in Section \ref{justification_BC}. The following proposition gives a link between $\overline{\textbf{V}}$ and $\underline{\textbf{U}}^{\mu} \cdot \textbf{N}^{\mu, \gamma}$ (the proof is a small adaptation of Proposition 4.2 in \cite{my_coriolis}).

\begin{prop}\label{mean_eq}
\noindent If \upshape $\left(\zeta, \Umus, \bm{\omega} \right)$ \itshape satisfy the Castro-Lannes system \eqref{castro_lannes_formulation}, we have

\upshape
\begin{equation*}
\underline{\textbf{U}}^{\mu} \cdot \textbf{N}^{\mu, \gamma} = - \mu \nabla^{\gamma} \cdot \left(h \overline{\textbf{V}} \right).
\end{equation*}
\itshape

\end{prop}

\noindent Then we get the first equation of the Boussinesq-Coriolis system from the first equation of \eqref{castro_lannes_formulation}. We also need an expansion of $\textbf{V}$ and $\text{w}$ with respect to $\mu$. We introduce the following operators

\begin{equation*}
T \left[\epsilon \zeta \right] f = \int_{z}^{\epsilon \zeta} \nabla^{\gamma} \nabla^{\gamma} \cdot  \int_{-1}^{z'} f \text{     and     } T^{\ast} \left[\epsilon \zeta \right] f = \left(  T \left[\epsilon \zeta \right] f \right)^{\ast},
\end{equation*}

\noindent In the following, we denote $T = T \left[\epsilon \zeta \right]$ and $T^{\ast} = T^{\ast} \left[\epsilon \zeta \right]$ when no confusion is possible.

\begin{prop}\label{equation_V}
\noindent In the Boussinesq regime $\mathcal{A}_{\text{boussi}}$, if \upshape $\left(\zeta, \Umus, \bm{\omega} \right)$ \itshape satisfy the Castro-Lannes system \eqref{castro_lannes_formulation}, we have

\upshape
\begin{equation*}
\begin{aligned}
&\textbf{V} = \overline{\textbf{V}} + \sqrt{\mu} \textbf{V}_{\text{sh}}^{\ast} + \mu  T^{\ast} \overline{\textbf{V}} + \mu^{\frac{3}{2}} T^{\ast} \textbf{V}_{\text{sh}}^{\ast} + \mathcal{O} \left( \mu^{2} \right),\\
&\underline{\textbf{V}} = \overline{\textbf{V}} - \sqrt{\mu} \textbf{Q} + \mu  \underline{T^{\ast} \overline{\textbf{V}}} - \mu^{\frac{3}{2}} \overline{T \textbf{V}_{\text{sh}}^{\ast}} + \mathcal{O} \left( \mu^{2} \right),
\end{aligned}
\end{equation*}
\itshape

\noindent where

 \upshape
\begin{equation*}
T^{\ast} \overline{\textbf{V}} = \frac{1}{2} \left(\frac{h^{2}}{3} - \left[ z+1 \right]^{2} \right) \nabla^{\gamma} \nabla^{\gamma} \cdot \overline{\textbf{V}} \text{   and   } \underline{T^{\ast} \overline{\textbf{V}}} =  - \frac{h^{2}}{3} \nabla^{\gamma} \nabla^{\gamma} \cdot \overline{\textbf{V}}.
\end{equation*}
\itshape

\noindent We also have

\upshape
\begin{equation*}
\begin{aligned}
&\text{w} =- \mu (z+1) \nabla^{\gamma} \overline{\textbf{V}} + \mu^{\frac{3}{2}} \int_{-1}^{z} \nabla^{\gamma} \cdot \textbf{V}_{\text{sh}}^{\ast} + \mathcal{O} \left( \mu^{2} \right),\\
&\underline{\text{w}} = - \mu h \nabla^{\gamma} \cdot \overline{\textbf{V}} + \mathcal{O} \left( \mu^{2} \right),
\end{aligned}
\end{equation*}
\itshape

\end{prop}

\begin{proof}
\noindent This proof is an adaptation of part 2.2 in \cite{Castro_Lannes_shallow_water}, Part 4.2 in  \cite{my_coriolis} and Section 2.1 in \cite{my_long_wave_corio}. First, using $\text{curl}^{\mu, \gamma} \; \textbf{U}^{\mu} = \mu \bm{\omega}$, we obtain that

\begin{equation*}
\sqrt{\mu} \bm{\omega}_{h} = \partial_{z} \textbf{V}^{\perp} - \nabla^{\gamma \perp} \text{w}.
\end{equation*}

\noindent Then, we consider the ansatz $\textbf{V} = \overline{\textbf{V}} + \sqrt{\mu} \textbf{V}_{1}$. By integrating the previous equation, we obtain

\begin{equation*}
\sqrt{\mu} \partial_{z} \textbf{V}_{1} = - \sqrt{\mu} \bm{\omega}_{h}^{\perp} + \nabla^{\gamma \perp} \text{w}.
\end{equation*}

\noindent Since $\overline{\textbf{V}_{1}} = 0$, we get

\begin{equation*}
\textbf{V}_{1} = \left(\int_{z}^{\epsilon \zeta} \bm{\omega}_{h}^{\perp} \right)^{\ast} - \frac{1}{\sqrt{\mu}} \left(\int_{z}^{\epsilon \zeta} \nabla^{\gamma} \text{w} \right)^{\ast}.
\end{equation*}

\noindent Secondly, using Proposition \ref{mean_eq} and the divergence-free assumption, we get 

\begin{equation}\label{exact_w}
\text{w} = - \mu \nabla^{\gamma} \cdot \left(\int_{-1}^{z} \textbf{V} \right).
\end{equation}

\noindent Then, gathering the previous two equality, we obtain

\begin{equation}\label{exact_epansion_V}
\textbf{V} = \overline{\textbf{V}} + \sqrt{\mu} \textbf{V}_{\text{sh}}^{\ast} + \mu  T^{\ast} \textbf{V}.
\end{equation}

\noindent Finally, the expansion of $\textbf{V}$ follows by applying the operator $Id + \mu  T^{\ast}$ to the previous equality. For the second equality, we notice that $\underline{T^{\ast} \textbf{V}_{\text{sh}}^{\ast}} = - \overline{T \textbf{V}_{\text{sh}}^{\ast}}$. The third and fourth equalities follow from the fact that $\overline{\textbf{V}}$ does not depend on $z$. The fifth equality are a consequence of Equalities \eqref{exact_w} and \eqref{exact_epansion_V}. Finally, the sixth equality follows from the fact that $\overline{\textbf{V}_{\text{sh}}^{\ast}} = 0$ and that $\epsilon=  \mathcal{O}(\mu)$.
\end{proof}

\noindent We can also get an expansion of $\partial_{t} \textbf{V}$ and $\partial_{t} \text{w}$.

\begin{prop}\label{equation_partial_t_V}
\noindent In the Boussinesq regime $\mathcal{A}_{\text{boussi}}$, if \upshape $\left(\zeta, \Umus, \bm{\omega} \right)$ \itshape satisfy the Castro-Lannes system \eqref{castro_lannes_formulation}, we have

\upshape
\begin{equation*}
\begin{aligned}
&\partial_{t} \left(\textbf{V} - \overline{\textbf{V}} - \sqrt{\mu} \textbf{V}_{\text{sh}}^{\ast} - \mu T^{\ast} \overline{\textbf{V}} - \mu^{\frac{3}{2}} T^{\ast} \textbf{V}_{\text{sh}}^{\ast} \right) = \mathcal{O} \left( \mu^{2} \right),\\
&\partial_{t} \left(\underline{\textbf{V}} - \overline{\textbf{V}} + \sqrt{\mu} \textbf{Q} - \mu \underline{T^{\ast} \overline{\textbf{V}}} + \mu^{\frac{3}{2}} \overline{T \textbf{V}_{\text{sh}}^{\ast}} \right) = \mathcal{O} \left( \mu^{2} \right),\\
&\partial_{t} \left( \underline{\text{w}} + \mu h \nabla^{\gamma} \overline{\textbf{V}} \right) = \mathcal{O} \left( \mu^{2} \right).
\end{aligned}
\end{equation*}
\itshape

\end{prop}

\begin{proof}
\noindent The result follows from Proposition \ref{mean_eq} and the equality

\begin{equation*}
\textbf{V} = \left(1- \mu T^{\ast} \right) \left( \overline{\textbf{V}} + \sqrt {\mu} \textbf{V}_{\text{sh}}^{\ast} \right) + \mu^{2} T^{\ast} T^{\ast} \textbf{V}.
\end{equation*}
\end{proof}

\noindent Since we can not express $\text{Q}$ and $\Vastsh$ with respect to $\zeta$ and $\overline{\textbf{V}}$, we need an evolution equation at order $\mathcal{O} \left( \mu^{\frac{3}{2}} \right)$ of these quantities. 

\begin{prop}\label{eq_Q}
\noindent In the Boussinesq regime $\mathcal{A}_{\text{boussi}}$, if \upshape $\left(\zeta, \Umus, \bm{\omega} \right)$ \itshape satisfy the Castro-Lannes system \eqref{castro_lannes_formulation}, then $\text{Q}$ satisfies the following equation

\upshape
\begin{equation*}
\partial_{t} \textbf{Q} + \epsilon \overline{\textbf{V}} \cdot \nabla^{\gamma} \textbf{Q} + \epsilon \textbf{Q} \cdot \nabla^{\gamma} \overline{\textbf{V}} + \frac{\epsilon}{\text{Ro} \sqrt{\mu}} \left(\overline{\textbf{V}} - \underline{\textbf{V}} \right)^{\perp} = \mathcal{O} \left( \mu^{\frac{3}{2}} \right),
\end{equation*}
\itshape

\noindent and $\Vastsh$ satisfies the equation

\upshape
\begin{equation*}
\partial_{t} \Vastsh + \epsilon \overline{\textbf{V}} \cdot \nabla^{\gamma} \Vastsh + \epsilon \Vastsh \cdot \nabla^{\gamma} \overline{\textbf{V}} - \epsilon \left[1+z \right] \left( \nabla^{\gamma} \cdot \overline{\textbf{V}} \right) \partial_{z}  \textbf{V}_{\text{sh}}^{\ast} + \frac{\epsilon}{\text{Ro} \sqrt{\mu}} \left(\textbf{V} - \overline{\textbf{V}} \right)^{\perp} = \mathcal{O} \left( \mu^{\frac{3}{2}} \right).
\end{equation*}
\itshape
\end{prop}

\begin{proof}
\noindent This proof is an adaptation of Part 2.3 in \cite{Castro_Lannes_shallow_water} and Part 2.2 in \cite{my_long_wave_corio}. Thanks to the horizontal component of the vorticity equation of the Castro-Lannes formulation \eqref{castro_lannes_formulation}, we get

\begin{equation*}
\partial_{t} \bm{\omega}_{h} + \epsilon \textbf{V} \cdot \nabla^{\gamma} \bm{\omega}_{h} + \frac{\epsilon}{\mu} \text{w} \partial_{z}  \bm{\omega}_{h} = \epsilon \bm{\omega}_{h} \cdot \nabla^{\gamma} \textbf{V} + \frac{\epsilon}{\sqrt{\mu}} \bm{\omega}_{z} \partial_{z} \textbf{V} + \frac{\epsilon}{\text{Ro} \sqrt{\mu}} \partial_{z} \textbf{V}.
\end{equation*}

\noindent Furthermore, since $\text{curl}^{\mu, \gamma} \; \textbf{U}^{\mu} = \mu \bm{\omega}$, we have

\begin{equation*}
\partial_{z} \textbf{V} = - \sqrt{\mu} \bm{\omega}_{h}^{\perp} + \mathcal{O} \left( \mu \right) \text{   and   } \bm{\omega}_{z} = \nabla^{\gamma \perp} \cdot \overline{\textbf{V}} + \mathcal{O} \left( \sqrt{\mu} \right).
\end{equation*}

\noindent Then, using Proposition \ref{equation_V}, we obtain

\small
\begin{equation*}
\partial_{t} \bm{\omega}_{h} + \epsilon \overline{\textbf{V}} \cdot \nabla^{\gamma} \bm{\omega}_{h} - \epsilon \left[ 1 + z \right] \left( \nabla^{\gamma} \cdot \overline{\textbf{V}} \right) \partial_{z}  \bm{\omega}_{h} - \epsilon \bm{\omega}_{h} \cdot \nabla^{\gamma} \overline{\textbf{V}} - \epsilon \left( \nabla^{\gamma \perp} \overline{\textbf{V}} \right) \bm{\omega}_{h}^{\perp} - \frac{\epsilon}{\text{Ro} \sqrt{\mu}} \partial_{z} \textbf{V} =  \mathcal{O} \left( \mu^{\frac{3}{2}} \right),
\end{equation*}
\normalsize

\noindent Then, integrating with respect to $z$, using the fact that $\partial_{t} \zeta + \nabla^{\gamma}  \cdot \left( h \overline{\textbf{V}} \right) = 0$, $\textbf{V}_{\text{sh}} = \int_{z}^{\epsilon \zeta} \bm{\omega}_{h}^{\perp}$ and $\textbf{Q}_{x} = \overline{\textbf{V}_{\text{sh}}^{\ast}}$ we get (see the computations in Part 2.3 in \cite{Castro_Lannes_shallow_water})

\begin{equation*}
\partial_{t} \textbf{V}_{\text{sh}} + \epsilon \overline{\textbf{V}} \cdot \nabla^{\gamma} \textbf{V}_{\text{sh}} + \epsilon \textbf{V}_{\text{sh}} \cdot \nabla^{\gamma} \overline{\textbf{V}} + \frac{\epsilon}{\text{Ro} \sqrt{\mu}} \left( \textbf{V} - \underline{\textbf{V}} \right)^{\perp} = \epsilon \left[1+z \right] \left( \nabla^{\gamma} \cdot \overline{\textbf{V}} \right) \partial_{z}  \textbf{V}_{\text{sh}} + \mathcal{O} \left( \mu^{\frac{3}{2}} \right).
\end{equation*}

\noindent and

\begin{equation*}
\partial_{t} \textbf{Q} + \epsilon \overline{\textbf{V}} \cdot \nabla^{\gamma} \textbf{Q} + \epsilon \textbf{Q} \cdot \nabla^{\gamma} \overline{\textbf{V}} + \frac{\epsilon}{\text{Ro} \sqrt{\mu}} \left( \overline{\textbf{V}} - \underline{\textbf{V}} \right)^{\perp} = \mathcal{O} \left( \mu^{\frac{3}{2}} \right).
\end{equation*}

\noindent Finally, the second equation follows from the fact that $\Vastsh = \textbf{V}_{\text{sh}} - \textbf{Q}$.
\end{proof}

\subsection{Full justification of the Boussinesq-Coriolis equations}\label{justification_BC}

\noindent We can now establish the Boussinesq-Coriolis equations under a weak Coriolis forcing. The Boussinesq-Coriolis equations \eqref{boussinesq_eq} are the following system

\begin{equation*}
\left\{
\begin{aligned}
&\partial_{t} \zeta + \nabla^{\gamma} \cdot h \overline{\textbf{V}} = 0,\\
&\left(1- \frac{\mu}{3} \nabla^{\gamma} \nabla^{\gamma} \cdot \right) \partial_{t} \overline{\textbf{V}} + \nabla^{\gamma} \zeta + \epsilon \overline{\textbf{V}} \cdot \nabla^{\gamma} \overline{\textbf{V}} + \frac{\epsilon}{\text{Ro}} \overline{\textbf{V}}^{\perp} = 0.
\end{aligned}
\right.
\end{equation*}

\noindent First, we show a consistency result.

\begin{prop}\label{constit_boussi}
\noindent In the  Boussinesq regime $\mathcal{A}_{Boussi}$, the Castro-Lannes equations \eqref{castro_lannes_formulation} are consistent at order $\mathcal{O}(\mu^{2})$ with the Boussinesq-Coriolis equations \eqref{boussinesq_eq} in the sense of Definition \ref{consistence}.
\end{prop}

\begin{proof}
\noindent The first equation of the Boussinesq-Coriolis equations is always satisfied for a solution of the Castro-Lannes formulation by Proposition \ref{mean_eq}. We recall that the second equation of the Castro-Lannes formulation is 

\begin{equation*}
\partial_{t} \Umus \hspace{-0.1cm} +  \hspace{-0.1cm} \nabla^{\gamma} \zeta  \hspace{-0.1cm} +  \hspace{-0.1cm} \frac{\epsilon}{2} \nabla^{\gamma}  \hspace{-0.1cm} \left\lvert \Umus \right\rvert^{2}  \hspace{-0.1cm} -  \hspace{-0.1cm} \frac{\epsilon}{2 \mu} \nabla^{\gamma}  \hspace{-0.1cm} \left[ \left( 1 +  \epsilon^{2} \mu \left\lvert \nabla^{\gamma} \zeta\right\rvert^{2} \right) \underline{\text{w}}^{2} \right]  \hspace{-0.1cm} +  \hspace{-0.1cm} \epsilon \hspace{-0.05cm} \left(\hspace{-0.05cm} \nabla^{\perp} \hspace{-0.05cm} \cdot \hspace{-0.05cm} \Umus \hspace{-0.05cm} \right) \hspace{0.05cm} \underline{\textbf{V}}^{\perp}  \hspace{-0.1cm} +  \hspace{-0.1cm} \frac{\epsilon}{\text{Ro}} \underline{\textbf{V}}^{\perp} = 0.
\end{equation*}

\noindent Thanks to Proposition \ref{equation_V}, we know that ($\epsilon = \mathcal{O} (\mu)$)

\begin{equation*}
\Umus = \underline{\textbf{V}} + \epsilon \underline{\text{w}} \nabla^{\gamma} \zeta = \underline{\textbf{V}} + \mathcal{O} \left( \mu^{2} \right) = \overline{\textbf{V}} - \sqrt{\mu} \textbf{Q} + \mu  \underline{T^{\ast} \overline{\textbf{V}}} - \mu^{\frac{3}{2}} \overline{T \textbf{V}_{\text{sh}}^{\ast}} + \mathcal{O} \left( \mu^{2} \right),
\end{equation*}

\noindent and 

\begin{align*}
\frac{\epsilon}{2} \nabla^{\gamma} \lver \Umus \rver ^{2} \hspace{-0.1cm} &= \hspace{-0.1cm} \epsilon \Umus \cdot \nabla^{\gamma} \Umus - \epsilon \left( \nabla^{\gamma \perp} \cdot \Umus \right) \textbf{U}^{\mu,\gamma \perp}_{\! \sslash}\\
&= \hspace{-0.1cm} \epsilon \overline{\textbf{V}} \cdot \nabla^{\gamma} \overline{\textbf{V}} - \epsilon \sqrt{\mu} \textbf{Q} \cdot \nabla^{\gamma} \overline{\textbf{V}} - \epsilon \sqrt{\mu}  \overline{\textbf{V}} \cdot \nabla^{\gamma}  \textbf{Q}- \epsilon \left( \hspace{-0.05cm} \nabla^{\gamma \perp} \cdot \Umus \hspace{-0.05cm} \right) \underline{\textbf{V}}^{\perp} + \mathcal{O} \left( \mu^{2} \right) \hspace{-0.1cm}.
\end{align*}

\noindent Furthermore, thanks to Proposition \ref{eq_Q} and Proposition \ref{equation_V}, we get $\left( \frac{\epsilon}{\text{Ro}} = \mathcal{O} \left( \sqrt{\mu} \right) \right)$

\begin{equation*}
\mu^{\frac{3}{2}} \partial_{t} \overline{T \textbf{V}_{\text{sh}}^{\ast}} = \mu^{\frac{3}{2}} \overline{T \partial_{t} \textbf{V}_{\text{sh}}^{\ast}} + \mathcal{O} \left( \mu^{2} \right) = - \mu^{\frac{3}{2}} \frac{\epsilon}{\text{Ro}} \overline{T \textbf{V}_{\text{sh}}^{\ast \perp}} + \mathcal{O} \left( \mu^{2} \right) = \mathcal{O} \left( \mu^{2} \right).
\end{equation*}

\noindent Finally, using Proposition \ref{equation_V}, Proposition \ref{eq_Q}, Proposition \ref{equation_partial_t_V} and  Remark \ref{control_remainder}, we obtain from the second equation of the Castro-Lannes formulation that

\begin{equation*}
\left(1- \frac{\mu}{3} \nabla^{\gamma} \nabla^{\gamma} \cdot \right) \partial_{t} \overline{\textbf{V}} + \nabla^{\gamma} \zeta + \epsilon \overline{\textbf{V}} \cdot \nabla^{\gamma} \overline{\textbf{V}} + \frac{\epsilon}{\text{Ro}} \overline{\textbf{V}}^{\perp} = \mathcal{O} \left( \mu^{2} \right).
\end{equation*}

\noindent Notice that all the terms that involve $\textbf{Q}$ disappear (this fact was pointed out in \cite{Castro_Lannes_shallow_water} and \cite{my_coriolis}).
\end{proof}

\begin{remark}\label{term_vort_strange}
\noindent In \cite{my_long_wave_corio}, the author points out the fact that under a strong Coriolis forcing $\left( \frac{\epsilon}{\text{Ro}} \leq 1\right)$, a new term appears in the Boussinesq-Coriolis equations. We would like to emphasize that this term is not present in this setting since we only study a weak Coriolis forcing $\left( \frac{\epsilon}{\text{Ro}} = \mathcal{O} \left(\sqrt{\mu} \right) \right)$.
\end{remark}

\noindent The purpose of this part is to fully justify the Boussinesq-Coriolis equations \eqref{boussinesq_eq}. First, we give a local wellposedness result of the Boussinesq-Coriolis equations. We define the energy space

\begin{equation*}
X^{N}_{\mu}(\RR) =  H^{N}(\RR) \times H^{N}(\RR) \times H^{N}(\RR),
\end{equation*}

\noindent endowed with the norm 

\begin{equation*}
\lver (\zeta,\textbf{V}) \rver_{X^{N}_{\mu}}^{2} =  \lver \zeta \rver^{2}_{H^{N}} +  \lver \textbf{V} \rver^{2}_{H^{N}} + \mu \lver \nabla^{\gamma} \cdot \textbf{V} \rver^{2}_{H^{N}}.
\end{equation*}

\begin{prop}\label{existence_boussi}
\noindent Let $N \geq 3$ and \upshape $ \left( \zeta_{0}, \overline{\textbf{V}}_{0} \right) \in X^{N}_{\mu}(\RR)$ \itshape. Suppose that \upshape$\left( \mu,\epsilon,\gamma,\text{Ro} \right) \in \mathcal{A}_{\text{boussi}}$\itshape. Assume that

\upshape
\begin{equation*}
\exists \, h_{\min}  > 0 \text{ ,  } 1+ \epsilon \zeta _{0} \geq h_{\min}.
\end{equation*}
\itshape

\noindent Then, there exists an existence time \upshape$T > 0$ \itshape and a unique solution \upshape$\left( \zeta,\overline{\textbf{V}} \right)$ \itshape on $[0,T]$ to the Boussinesq-Coriolis equations \eqref{boussinesq_eq} with initial data \upshape$\left( \zeta_{0}, \overline{\textbf{V}}_{0} \right)$ \itshape. Moreover, \small \upshape$\left( \zeta, \overline{\textbf{V}} \right) \in \mathcal{C} \left([0,T]; X^{N}_{\mu}(\RR) \right)$ \itshape \normalsize and

\upshape
\begin{equation*}
T =  \frac{T_{0}}{\mu} \text{    ,    }  \frac{1}{T_{0}} = c^{1} \text{   and   }\underset{[0,T]}{\max} \lver \left(\zeta, \overline{\textbf{V}} \right)(t, \cdot) \rver_{X^{N}_{\mu}} = c^{2},
\end{equation*}
\itshape

\noindent with \upshape $c^{j} = C \left( \mu_{0}, \frac{1}{h_{\min}}, \lver \left( \zeta_{0}, \overline{\textbf{V}}_{0} \right) \rver_{X^{N}_{\mu}} \right)$.\itshape

\end{prop}

\begin{proof}
\noindent This proof is a small adaptation of the proof of Proposition 6.7 in \cite{Lannes_ww}. We only give the energy estimates. We assume that $\left( \zeta, \overline{\textbf{V}} \right)$ solves \eqref{boussinesq_eq} on $\left[0,\frac{T_{0}}{\mu} \right]$ and that

\begin{equation*}
1+\epsilon \zeta \geq \frac{h_{\min}}{2} \text{   on   } \left[0,\frac{T_{0}}{\mu} \right].
\end{equation*}

\noindent We denote $U = \left(\zeta,\overline{\textbf{V}} \right)^{t}$. We introduce the symmetric matrix operator

\begin{equation*}
S(U) = \begin{pmatrix} 1 & 0  \\ 0 & h I_{2} - \mu \frac{1}{3} \nabla^{\gamma} \left( h \nabla^{\gamma} \cdot \right) \end{pmatrix}
\end{equation*}
 
\noindent and the associated energy

\begin{equation*}
\mathcal{E}^{N}(U) = \frac{1}{2} \underset{\lver \alpha \rver \leq N}{\sum} \left(S(U) \partial^{\alpha} U, \partial^{\alpha} U \right)_{2}.
\end{equation*}

\noindent Remark that there exists $c_{1},c_{2} = C \left( \frac{1}{h_{\min}}, \lver h \rver_{L^{\infty}} \right)$ such that

\begin{equation*}
c_{1} \lver \nabla^{\gamma} \cdot \textbf{V} \rver^{2}_{2} \leq \left( - \frac{1}{3} \nabla^{\gamma}  \left(h \nabla^{\gamma} \cdot \overline{\textbf{V}} \right) , \overline{\textbf{V}} \right)_{2}  \leq c_{2} \lver \nabla^{\gamma} \cdot \textbf{V} \rver^{2}_{2}.
\end{equation*}

\noindent We also notice that for $|\alpha| = N$,

\begin{equation*}
\frac{d}{dt} \left( S(U) \partial^{\alpha} U \right) = \begin{pmatrix} \partial^{\alpha} \partial_{t} \zeta \\ h \left( 1 - \frac{\mu}{3} \nabla^{\gamma} \nabla^{\gamma} \cdot \right) \partial^{\alpha} \partial_{t} \overline{\textbf{V}} \end{pmatrix} - \begin{pmatrix} 0 \\ \epsilon \frac{\mu}{3}  \left( \nabla^{\gamma} \cdot \partial^{\alpha} \partial_{t} \overline{\textbf{V}} \right) \nabla^{\gamma} \zeta \end{pmatrix} + \epsilon \; l.o.t.
\end{equation*}

\noindent and that, denoting $\Delta^{\gamma} = \nabla^{\gamma} \cdot \nabla^{\gamma}$,

\small
\begin{equation*}
\mu \lver \nabla^{\gamma} \cdot \partial_{t} \overline{\textbf{V}} \rver_{H^{N}} \leq \lver \left(1 - \frac{\mu}{3} \Delta^{\gamma} \right)^{-1} \mu \nabla^{\gamma} \cdot \left( \nabla^{\gamma} \zeta + \epsilon \overline{\textbf{V}} \cdot \nabla^{\gamma} \overline{\textbf{V}} + \frac{\epsilon}{\text{Ro}} \overline{\textbf{V}}^{\perp} \right) \rver_{H^{N}} \leq C \left( \mu_{0}, \mathcal{E}^{N}(U) \right).
\end{equation*}
\normalsize

\noindent Then, after some computations we obtain ($\epsilon = \mathcal{O} \left( \mu \right)$)

\begin{equation*}
\frac{d}{dt} \mathcal{E}^{N}(U) \leq \mu C \left(\mathcal{E}^{N}(U) \right) \mathcal{E}^{N}(U)
\end{equation*}

\noindent and the result follows from Gr\"onwall's inequality.
\end{proof}

\noindent We also have a stability result for the Boussinesq-Coriolis equations \eqref{boussinesq_eq}.

\begin{prop}\label{stability_boussi}
\noindent Let the assumptions of Proposition \ref{existence_boussi} be satisfied. Suppose that there exists  \upshape$\left( \tilde{\zeta},  \tilde{\textbf{V}} \right) \in \mathcal{C} \left(\left[0,\frac{T_{0}}{\mu} \right]; X^{N}_{\mu}(\RR) \right)$ \itshape satisfying

\upshape
\begin{equation*}
\left\{
\begin{aligned}
&\partial_{t} \tilde{\zeta} + \nabla^{\gamma} \cdot \tilde{h} \tilde{\textbf{V}} = R_{1},\\
&\left(1- \frac{\mu}{3} \nabla^{\gamma} \nabla^{\gamma} \cdot \right) \partial_{t} \tilde{\textbf{V}} + \nabla^{\gamma} \tilde{\zeta} + \epsilon \tilde{\textbf{V}} \cdot \nabla^{\gamma} \tilde{\textbf{V}} + \frac{\epsilon}{\text{Ro}} \tilde{\textbf{V}}^{\perp} = R_{2}.
\end{aligned}
\right.
\end{equation*}
\itshape

\noindent where $\tilde{h} = 1 + \epsilon \tilde{\zeta}$ and with $R = (R_{1}, R_{2}) \in L^{\infty} \left( \left[0, \frac{T_{0}}{\mu} \right] ;H^{N-1}(\RR) \right)$. Then, if we denote \upshape $\mathfrak{e} = \left( \zeta, \overline{\textbf{V}} \right) - \left( \tilde{\zeta}, \tilde{\textbf{V}} \right)$ \itshape where \upshape $U = \left( \zeta, \textbf{V} \right)$ \itshape is the solution given in Proposition \ref{existence_boussi}, we have

\upshape
\begin{equation*}
\lver \mathfrak{e}(t) \rver_{X^{N-1}_{\mu}} \leq c_{1} \left(\lver \mathfrak{e}_{|t=0} \rver_{X^{N-1}_{\mu}} + t \lver R \rver_{L^{\infty} \left( \left[0, \frac{T_{0}}{\mu} \right] ; H^{N-1} \right)} \right),
\end{equation*}
\normalsize

\noindent where 

\begin{equation*}
c_{1} = C \left(\mu_{0}, \frac{1}{h_{\min}}, \lver \left( \zeta_{0}, \overline{\textbf{V}}_{0} \right) \rver_{X^{N}_{\mu}}, \lver \left( \tilde{\zeta}, \tilde{\textbf{V}} \right) \rver_{L^{\infty} \left( \left[0, \frac{T_{0}}{\mu} \right] ; X^{N}_{\mu} \right)}, \lver R \rver_{L^{\infty} \left( \left[0, \frac{T_{0}}{\mu} \right] ; H^{N-1} \right)} \right).
\end{equation*}

\end{prop}

\begin{proof}
\noindent We proceed as in Proposition \ref{existence_boussi}. We define the energy

\begin{align*}
\mathcal{F}^{N-1}(\mathfrak{e}) =  \frac{1}{2} \underset{\lver \alpha \rver \leq N-1}{\sum} \left(S(U) \partial^{\alpha} \mathfrak{e}, \partial^{\alpha} \mathfrak{e} \right)_{2}.
\end{align*}

\noindent After some computations, we get

\small
\begin{equation*}
\frac{d}{dt} \mathcal{F}^{N-1}(\mathfrak{e}) \leq \left( \lver R \rver_{H^{N-1}} + \mu C \left(\mu_{0}, \frac{1}{h_{\min}}, \lver U \rver_{X^{N}_{\mu}}, \lver \left( \tilde{\zeta}, \tilde{\textbf{V}} \right) \rver_{X^{N}_{\mu}}, \lver R \rver_{H^{N-1}} \right) \lver \mathfrak{e} \rver_{X^{N-1}_{\mu}} \right) \lver \mathfrak{e} \rver_{X^{N-1}_{\mu}}.
\end{equation*}
\normalsize

\noindent Then the result follows from Gronwall's inequality.
\end{proof}

\noindent We can now rigorously justify the Boussinesq-Coriolis equations. We recall that the operator $\overline{\textbf{V}} [\epsilon \zeta](\Umus,\bm{\omega})$ is defined in \eqref{V_average}.

\begin{thm}\label{boussi_water_wave_compare}
\noindent Let \upshape$\textbf{N} \geq 12$ and \upshape$\left(\mu, \epsilon,\gamma,\text{Ro} \right) \in \mathcal{A}_{Boussi}$ \itshape. We assume that we are under the assumptions of Theorem \ref{existence_ww}. We define the following quantities

\upshape 
\begin{equation*}
\overline{\textbf{V}}_{0} = \overline{\textbf{V}} [\epsilon \zeta_{0}]((\Umus)_{0},\bm{\omega}_{0}) \text{ ,  } \overline{\textbf{V}}  = \overline{\textbf{V}} [\epsilon \zeta](\Umus,\bm{\omega}).
\end{equation*}
\itshape 

\noindent Then, there exists a time $T > 0$ such that

\medskip

\noindent (i) $T$ has the form

\begin{equation*}
T = \frac{T_{0}}{\max \left(\mu, \frac{\epsilon}{\text{Ro}} \right)} ,\text{ and }  \frac{1}{T_{0}} = c^{1}.
\end{equation*}
\itshape

\medskip

\noindent (ii) There exists a unique classical solution \small \upshape$\left( \zeta_{B}, \overline{\textbf{V}}_{B} \right)$ \itshape \normalsize of \eqref{boussinesq_eq} with the initial data \upshape$\left(\zeta_{0}, \overline{\textbf{V}}_{0} \right)$ \itshape  on $\left[ 0,T \right]$.

\medskip

\noindent (iii)  There exists a unique classical solution \upshape$\left(\zeta, \Umus, \bm{\omega} \right)$ \itshape of System \eqref{castro_lannes_formulation} with initial data \upshape$\left(\zeta_{0}, (\Umus)_{0}, \bm{\omega}_{0} \right)$ \itshape on $\left[ 0,T \right]$.

\medskip

\noindent (iv) The following error estimate holds, for $0 \leq t \leq T$,

\upshape
\begin{equation*}
\lver \left(\zeta, \overline{\textbf{V}} \right) - \left(\zeta_{B},\overline{\textbf{V}}_{B} \right) \rver_{L^{\infty}([0,t] \times \RR)} \leq \mu^{2} t\, c^{2},
\end{equation*}
\itshape

\noindent with \upshape $c^{j} = C \left(A, \mu_{0}, \frac{1}{h_{\min}}, \frac{1}{\mathfrak{a}_{\min}} \right)$.\itshape
\end{thm}

\noindent Therefore, in the Boussinesq regime $\mathcal{A}_{Boussi}$ a solution of the water waves system \eqref{castro_lannes_formulation} remains close to a solution of the Boussinesq-Coriolis equations \eqref{boussinesq_eq} over a time $\mathcal{O} \left(\frac{1}{\sqrt{\mu}} \right)$ with an accuracy of order $\mathcal{O} \left(\mu^{\frac{3}{2}} \right)$.

\begin{remark}
\noindent Notice that if one considers a solution of a system and wants to show that this solution remains close to a solution of the waves equations over times $\mathcal{O} \left(\frac{1}{\sqrt{\mu}} \right)$ with an accuracy of order $\mathcal{O} \left(\mu^{\frac{3}{2}} \right)$, it is sufficient to compare this solution with a solution of the Boussinesq-Coriolis equations \eqref{boussinesq_eq}. We use this strategy in the following.
\end{remark}

\section{The KP approximation}\label{KP_approx}

\noindent In this section, we consider the KP (Kadomtsev-Petviashvili) approximation under a weak Coriolis forcing. We assume that $\epsilon = \mu$ (long waves) and $\gamma  =\sqrt{\mu}$ (weakly transverse effects). We consider two different regimes. First, if $\frac{\epsilon}{\text{Ro}} = \sqrt{\mu}$ (weak rotation), we derive the rotation-modified KP equation \eqref{RKP_eq}. Then, if $\frac{\epsilon}{\text{Ro}} = \mu$ (very weak rotation), we derive the KP equation \eqref{KP_eq}. We refer to \cite{grimshaw_KP} for more physical explanations about these two models (see also \cite{grimshaw_survey,germain-renouard,grimshaw_intern}).
 
\subsection{Weak rotation, the rotation-modified KP equation}\label{model_RKP}

\noindent In the irrotational setting, the KP equation provides a good approximation of the water waves equation under the assumption that $\epsilon$ and $\mu$ have the same order and that $\gamma$ and $\sqrt{\mu}$ have the same order (see \cite{Saut_Lannes_KP} or Part 7.2 in \cite{Lannes_ww}). When a Coriolis forcing is taken into account, Grimshaw and Melville (\cite{grimshaw_KP}) derived an equation for long waves, which is an adaptation of the KP equation

\begin{equation}\label{RKP_eq}
\partial_{\xi} \left( \partial_{\tau} k +  \frac{3}{2} k \partial_{\xi} k + \frac{1}{6} \partial_{\xi}^{3} k \right) + \frac{1}{2} \partial_{yy} k = \frac{1}{2} k.
\end{equation}

\noindent This equation is called the rotation-modified KP equation or Grimshaw-Melville equation in the physics literature. Notice that this equation was originally derived for internal water waves (\cite{grimshaw_KP,germain-renouard}). In this part, we fully justify this equation. Inspired by \cite{grimshaw_KP,germain-renouard}, we consider the asymptotic regime

\begin{equation*}
\mathcal{A}_{\text{RKP}} = \left\{ \left(\mu, \epsilon, \gamma, \text{Ro} \right), 0 \leq \mu \leq \mu_{0}, \epsilon = \mu, \gamma = \sqrt{\mu}, \frac{\epsilon}{\text{Ro}} = \sqrt{\mu} \right\}.
\end{equation*}

\noindent Then, the Boussinesq-Coriolis equations become  ($\gamma=\sqrt{\mu}$)

\begin{equation}\label{boussi_weakrot}
\left\{
\begin{aligned}
&\partial_{t} \zeta + \nabla^{\gamma} \cdot \left( [1+\mu \zeta] \overline{\textbf{V}} \right) = 0,\\
&\left(1- \frac{\mu}{3} \nabla^{\gamma} \nabla^{\gamma} \cdot \right) \partial_{t} \overline{\textbf{V}} + \nabla^{\gamma} \zeta + \mu \overline{\textbf{V}} \cdot \nabla^{\gamma} \overline{\textbf{V}} + \sqrt{\mu} \overline{\textbf{V}}^{\perp} = 0.
\end{aligned}
\right.
\end{equation}

\noindent In the following, we denote $\textbf{V} = (u,v)^{t}$. Our strategy is similar to the one used in \cite{my_long_wave_corio} to fully justify the Ostrovsky equation. We consider an expansion of $\left( \zeta, \textbf{V}\right)$ with respect to $\mu$. Inspired by \cite{Saut_Lannes_KP} or Part 7.2 in \cite{Lannes_ww}, we seek an approximate solution $\left(\zeta_{app}, u_{app}, v_{app} \right)$ of \eqref{boussi_weakrot} in the form

\begin{equation}\label{ansatz_RKP}
\begin{aligned}
&\zeta_{app}(t,x,y) = k(x-t,y,\mu t) + \mu \zeta_{(1)}(t,x,y,\mu t),\\
&u_{app}(t,x,y) = k(x-t,y,\mu t) + \mu u_{(1)}(t,x,y,\mu t),\\
&v_{app}(t,x,y) = \sqrt{\mu} v_{(1/2)}(t,x,y,\mu t)
\end{aligned}
\end{equation}

\noindent where $k = k(\xi,\tau)$ is a traveling water wave modulated by a slow time variable and the others terms are correctors. In the following, we denote by $\tau$ the variable associated to the slow time variable $\mu t$. Plugging the ansatz into Sytem \eqref{boussi_weakrot}, we obtain

\small
\begin{equation}\label{boussi_weak_app}
\left\{
\begin{aligned}
&\partial_{t} \zeta_{app} + \nabla^{\gamma} \cdot \left( [1+\mu \zeta_{app}] \overline{\textbf{V}}_{app} \right) = \mu R^{1}_{(1)} + \mu^{2} R_{1},\\
&\left(1- \frac{\mu}{3} \nabla^{\gamma} \nabla^{\gamma} \cdot \right) \partial_{t} \overline{\textbf{V}}_{app} + \nabla^{\gamma} \zeta_{app} + \mu \overline{\textbf{V}}_{app} \cdot \nabla^{\gamma} \overline{\textbf{V}}_{app} + \sqrt{\mu} \overline{\textbf{V}}_{app}^{\perp} = \sqrt{\mu} R^{2}_{(1/2)} + \mu R^{2}_{(1)} + \mu^{\frac{3}{2}} R_{2}.
\end{aligned}
\right.
\end{equation}
\normalsize

\noindent where

\small
\begin{equation*}
\begin{aligned}
&R_{(1)}^{1} = \partial_{t} \zeta_{(1)} + \partial_{x} u_{(1)} + \partial_{\tau} k + 2 k \partial_{\xi} k  + \partial_{y} v_{(1/2)},\\
&R_{(1/2)}^{2} = \begin{pmatrix} 0 \\ \partial_{t} v_{(1/2)} + \partial_{y} k + k \end{pmatrix} \text{   and   } R_{(1)}^{2} = \begin{pmatrix} \partial_{t} u_{(1)} + \partial_{x} \zeta_{(1)} + \partial_{\tau} k +  \frac{1}{3} \partial^{3}_{\xi}  k + k \partial_{\xi} k - v_{(1/2)} \\ 0 \end{pmatrix},\\
\end{aligned}
\end{equation*}
\normalsize

\noindent and 

\small
\begin{equation}\label{remainder}
\begin{aligned}
&R_{1} = \partial_{\tau} \zeta_{(1)}  + \partial_{x} \left(k u_{(1)} + k \zeta_{(1)} + \mu \zeta_{(1)} u_{(1)} \right) + \partial_{y} ((k + \mu \zeta_{(1)}) v_{(1/2)}) ,\\
&R_{2} = \left(\sqrt{\mu} R_{2,1}, R_{2,2} \right)
\end{aligned}
\end{equation}
\normalsize

\noindent with

\small
\begin{align*}
&R_{2,1} =  \partial_{\tau} u_{(1)} - \frac{1}{3} \partial_{\xi}^{2} \partial_{\tau} k - \frac{1}{3} \partial_{x}^{2} \partial_{t} u_{(1)}  - \mu \frac{1}{3} \partial_{x}^{2} \partial_{\tau} u_{(1)} + \partial_{x} \left(k u_{(1)} \right) + \mu u_{(1)} \partial_{x} u_{(1)}\\
&\hspace{3cm} - \frac{1}{3} \partial^{3}_{xyt} v_{(1/2)} - \frac{\mu}{3} \partial^{3}_{xy \tau} v_{(1/2)} + v_{(1/2)} \partial_{y} \left(k + \mu u_{(1)} \right),\\
&R_{2,2} = \partial_{\tau} v_{(1/2)} + \partial_{y} \zeta_{(1)} + k \partial_{x} v_{(1/2)}  + u_{(1)} + \frac{1}{3} \partial_{y} \partial_{\xi}^{2} k + \mu u_{(1)} \partial_{x} v_{(1/2)} + \mu v_{(1/2)} \partial_{y} v_{(1/2)}  \\
&\hspace{3cm} - \frac{\mu}{3} \left( \partial^{3}_{y x \tau} k + \partial^{3}_{y x t} u_{(1)} + \partial_{y}^{2} \partial_{t} v_{(1/2)} + \mu \partial_{y x \tau} u_{(1)} + \mu \partial_{y}^{2} \partial_{\tau} v_{(1/2)} \right) .
\end{align*}
\normalsize

\noindent Then, the strategy is to choose $(k,v_{(1/2)})$ such that, for all $(x,y) \in \RR$, $t \in \left[0, \frac{T}{\mu} \right]$ and $\tau \in \left[0,T \right]$, 

\begin{equation*}
R_{(1)}^{1}(t,x,y,\tau) = 0 \text{   and   } R_{(1/2)}^{2}(t,x,y,\tau) = R_{(1)}^{2}(t,x,y,\tau) = 0.
\end{equation*}

\begin{remark}
\noindent As noticed in Part 7.2.2 in \cite{Lannes_ww}, we should a priori add $\sqrt{\mu} \zeta_{(1/2)}(t,x,y,\hspace{-0.02cm}\mu t\hspace{-0.05cm})$, $\sqrt{\mu} u_{(1/2)}(t,x,y,\mu t)$, $v_{(0)}(t,x,y,\mu t)$, and $\mu v_{(1)}(t,x,y,\mu t)$ to the ansatz \eqref{ansatz_RKP} for $\zeta_{app}$, $u_{app}$ and $v_{app}$ respectively. But, it leads to $\zeta_{(1/2)} = u_{(1/2)} = v_{(0)} = v_{(1)} = 0$ if these quantities are initially zero.
\end{remark}

\noindent We focus first on the condition $R_{(1/2)}^{2}(t,x,y,\tau) = 0$. Assuming that $v_{(1/2)}$ and $k$ vanish at $x=\infty$, this condition  is equivalent to the equation

\begin{equation*}
\partial_{t} \partial_{x} v_{(1/2)}(t,x,y,\tau) + \partial_{\xi} k(x-t,y,\tau) + \partial^{2}_{\xi y} k(x-t,y,\tau) = 0.
\end{equation*}

\noindent Then, using the fact that $\partial_{t} (k(x-t,y,\tau)) = - \partial_{\xi} k(x-t,y,\tau)$, we can integrate with respect to $t$ and we get

\small
\begin{equation*}
\partial_{x} v_{(1/2)}(t,x,y,\tau) = \partial_{x} v^{0}_{(1/2)}(x,y) + k(x-t,y,\tau) + \partial_{y} k(x-t,y,\tau) - k^{0}(x,y) - \partial_{y} k^{0}(x-t,y,\tau),
\end{equation*}
\normalsize

\noindent where $k^{0}$ and $v^{0}_{(1/2)}$ are the initial data of $k$ and $v_{(1/2)}$ respectively. Then, assuming that $k(\cdot,\tau) \in \partial_{x} H^{N}(\RR)$ for all $\tau \in [0,T]$ (see  \eqref{partialx_HN}), we obtain 

\begin{align*}
v_{(1/2)}(t,x,y,\tau) = &v^{0}_{(1/2)}(x,y) +  \partial_{x}^{-1} k(x-t,y,\tau) + \partial_{x}^{-1} \partial_{y} k(x-t,y,\tau)\\
&- \partial_{x}^{-1} k^{0}(x,y) - \partial_{x}^{-1} \partial_{y} k^{0}(x-t,y,\tau),
\end{align*}

\noindent Secondly, we study the conditions $R_{(1)}^{1} = R_{(1)}^{2} = 0$. Denoting $w_{\pm} = \zeta_{(1)} \pm u_{(1)}$, we obtain

\small
\begin{equation}\label{wave_RKP}
\begin{aligned}
&\left(\partial_{t} + \partial_{x} \right) w_{+} + \left( 2 \partial_{\tau} k + 3 k \partial_{\xi} k + \frac{1}{3} \partial_{\xi}^{3} k + \partial_{\xi}^{-1} \partial_{y}^{2} k - \partial_{\xi}^{-1} k \right) (x-t,\tau) + F^{1}_{0} = 0,\\
&\left(\partial_{t} - \partial_{x} \right) w_{-} + \left(k \partial_{\xi} k - \frac{1}{3} \partial_{\xi}^{3} k + \partial_{\xi}^{-1} \partial_{y}^{2} k + 2 \partial_{\xi}^{-1} \partial_{y} k + \partial_{\xi}^{-1} k \right) (x-t,\tau) + F^{2}_{0} = 0,\\
\end{aligned}
\end{equation}
\normalsize

\noindent where 

\begin{align*}
&F^{1}_{0} =  \partial_{y} v_{(1/2)}^{0} - v_{(1/2)}^{0} + \partial_{\xi}^{-1} k^{0} -  \partial_{\xi}^{-1} \partial_{y}^{2} k^{0},\\
&F^{2}_{0} = \partial_{y} v_{(1/2)}^{0} + v_{(1/2)}^{0} - \partial_{\xi}^{-1} k^{0} -  \partial_{\xi}^{-1} \partial_{y}^{2} k^{0} - 2 \partial_{\xi}^{-1} \partial_{y} k^{0}.
\end{align*}

\noindent The following Lemma (see Lemma 7.20 in \cite{Lannes_ww} or Lemma 2 in \cite{Saut_Lannes_KP}) gives us a Condition to control $\zeta_{(1)}$ and $u_{(1)}$.

\begin{lemma}\label{control_diff_speed}
\noindent Let $c_{1} \neq c_{2}$. Let $k_{1}, k_{2} \in H^{2}(\RR)$ with $k_{2} = \partial_{x} K_{2}$ and $K_{2} \in H^{3}(\RR)$. We consider the unique solution $k$ of

\begin{equation*}
\left\{
\begin{aligned}
&(\partial_{t} + c_{1} \partial_{x} )k = k_{1}(x-c_{1} t,y) + k_{2}(x-c_{2} t,y),\\
&k_{|t=0} = 0.
\end{aligned}
\right.
\end{equation*}

\noindent Then, $\underset{t \shortrightarrow \infty}{\lim} \lver \frac{1}{t} k(t, \cdot) \rver_{H^{2}} = 0$ if and only if $k_{1} \equiv 0$ and in that case

\begin{equation*}
\lver k(t, \cdot) \rver_{H^{2}} \leq C \frac{t}{1+t} \lver K_{2} \rver_{H^{3}}.
\end{equation*}
\end{lemma}

\noindent Hence, since we want to avoid a linear growth of the solution of \eqref{wave_RKP}, we must  impose  

\begin{equation}\label{RKP_eq_k}
\partial_{\tau} k + \frac{3}{2} k \partial_{\xi} k + \frac{1}{6} \partial_{\xi}^{3} k + \frac{1}{2} \partial_{\xi}^{-1} \partial_{y}^{2} k - \frac{1}{2} \partial_{\xi}^{-1} k = 0
\end{equation}

\noindent which is the the rotation-modified KP equation defined in \eqref{RKP_eq}. In the following, we provide a local existence result for this equation. This proposition generalizes Theorem 1.1 in \cite{soliton_RKP}.

\begin{prop}\label{existence_RKP}
\noindent Let $N \geq 4$ and $k_{0} \in \partial_{x} H^{N}(\RR)$. Then, there exists a time $T > 0$ and a unique solution $k \in \mathcal{C} \left([0,T]; \partial_{x} H^{N}(\RR) \right)$ to the rotation-modified KP equation \eqref{RKP_eq_k} and one has

\begin{equation*}
\lver \partial_{\xi}^{-1} k(t,\cdot) \rver_{H^{N}} \leq C \left(T, \lver \partial_{\xi}^{-1} k_{0} \rver_{H^{N}} \right).
\end{equation*}

\noindent Furthermore, if $k_{0}, \partial_{y}^{2} k_{0} \in \partial_{x}^{2} H^{N-2}(\RR)$, then $k \in \mathcal{C} \left([0,T]; \partial_{x}^{2} H^{N-2}(\RR) \right)$ and one has

\begin{align*}
&\lver \partial_{\xi}^{-2} k(t,\cdot) \rver_{H^{N-2}} \leq C \left(T, \lver \partial_{\xi}^{-2} k_{0} \rver_{H^{N-2}}, \lver \partial_{\xi}^{-2} \partial_{y}^{2} k_{0} \rver_{H^{N-2}}, \lver \partial_{\xi} k_{0} \rver_{H^{N}} \right).
\end{align*}

\noindent Finally, if $N \geq 6$ and $k_{0}, \partial_{y}^{2} k_{0} \in \partial_{x}^{2} H^{N-4}(\RR)$, then $\partial_{y}^{2} k \in \mathcal{C} \left([0,T]; \partial_{x}^{2} H^{N-4}(\RR) \right)$ and one has

\begin{align*}
&\lver \partial_{\xi}^{-2} k(t,\cdot) \rver_{H^{N-4}} \leq C \left(T, \lver \partial_{\xi}^{-2} k_{0} \rver_{H^{N-4}}, \lver \partial_{\xi}^{-2} \partial_{y}^{2} k_{0} \rver_{H^{N-4}}, \lver \partial_{\xi}^{-1} k_{0} \rver_{H^{N}} \right).
\end{align*}
\end{prop}

\begin{proof}
\noindent The first point follows from Theorem 1.1 in \cite{soliton_RKP}. We only have to prove the second and the third points. This proof is similar to the proof of Lemma 7.22 in \cite{Lannes_ww} for the KP equation and the proof of Proposition 3.8 in \cite{my_long_wave_corio} for the Ostrovsky equation. In the following, we denote by $S(t)$ the semi-group of the linearized rotation-modified KP equation

\begin{equation*}
\partial_{\tau} k + \frac{1}{6} \partial_{\xi}^{3} k + \frac{1}{2} \partial_{\xi}^{-1} \partial_{y}^{2} k - \frac{1}{2} \partial_{\xi}^{-1} k = 0.
\end{equation*}

\noindent One can check that this semi-group acts unitary on $H^{N}(\RR)$. We also define $\tilde{k} = \partial_{\tau} k$. We can check that

\begin{equation*}
\partial_{\tau} \tilde{k} + \frac{1}{6} \partial_{\xi}^{3} \tilde{k} + \frac{1}{2} \partial_{\xi}^{-1} \partial_{y}^{2} \tilde{k} - \frac{1}{2} \partial_{\xi}^{-1} \tilde{k} + \frac{3}{2} \partial_{\xi} \left( \tilde{k} k \right) = 0.
\end{equation*}

\noindent Using the Duhamel's formula we obtain

\begin{equation*}
\partial_{\xi}^{-1} \tilde{k}(\tau) = S(t) \partial_{\xi}^{-1} \tilde{k}_{0} - \frac{3}{2} \int_{0}^{\tau} S(t-s) \left( k(s) \tilde{k}(s)\right) ds.
\end{equation*}

\noindent We can see by product estimates that $\partial_{\xi}^{-1} \tilde{k}_{0}$, $k(s) \tilde{k}(s) \in H^{N-4}(\RR)$ and then that $\tilde{k} \in \mathcal{C} \left([0,T]; \partial_{x} H^{N-3}(\RR) \right)$. Then, we consider the following equality

\begin{equation*}
\frac{1}{2} \left(1- \partial_{y}^{2} \right) \partial_{\xi}^{-2} k = \partial_{\xi}^{-1} \tilde{k} + \frac{3}{4}  k^{2} + \frac{1}{6} \partial_{\xi}^{2} k,
\end{equation*}

\noindent For the second point, we get that $(1-\partial_{\xi}^{2} - \partial_{y}^{2}) \partial_{\xi}^{-2} k \in H^{N-4}(\RR)$ and the result follows easily. For the third point, we obtain from the second point that $\partial_{y}^{2} \partial_{\xi}^{-2} k \in H^{N-4}(\RR)$.
\end{proof}

\noindent We can now rigorously justify the rotation-modified KP equation. The following theorem is the main theorem of this part.

\begin{thm}\label{full_just_RKP}
\noindent Let $k^{0} \in \partial_{x}^{2} H^{12}(\RR)$ such that $1+\epsilon k^{0} \geq h_{\min} > 0$ and $v^{0} \in \partial_{x} H^{8}(\RR)$. Suppose that \upshape$\left( \mu,\epsilon,\gamma,\text{Ro} \right) \in \mathcal{A}_{\text{RKP}}$\itshape. Then, there exists a time $T_{0} > 0$, such that we have

\medskip

\noindent (i) a unique classical solution \upshape$\left(\zeta_{B},u_{B}, v_{B} \right)$ \itshape of \eqref{boussi_weakrot} with initial data \upshape$\left(k^{0}, k^{0}, \sqrt{\mu} v^{0} \right)$ \itshape on $\left[ 0, \frac{T_{0}}{\mu} \right]$.

\medskip

\noindent (ii)  a unique classical solution \upshape$k$ \itshape of \eqref{RKP_eq_k} with initial data \upshape$k^{0}$ \itshape on $\left[ 0, T_{0} \right]$.

\medskip 

\noindent (iii) If we define $\left(\zeta_{RKP}, u_{RKP} \right)(t,x,y) = \left( k(x-t,y,\mu t),k(x-t,y,\mu t ) \right)$ we have the following error estimate for all $0 \leq t \leq \frac{T_{0}}{\mu}$,

\begin{equation*}
\lver \left(\zeta_{B},u_{B} \right) - \left(\zeta_{RKP},u_{RKP} \right) \rver_{L^{\infty}([0,t] \times \RR)} \leq  C \frac{\mu t}{1+t} (1+ \sqrt{\mu} t) 
\end{equation*}
\itshape

\noindent where $C = C\left(\frac{1}{h_{\min}}, \mu_{0}, \lver \partial_{x}^{-2} k^{0} \rver_{H^{12}}, \lver \partial_{x}^{-1} v^{0} \rver_{H^{8}} \right)$.

\end{thm}

\begin{proof}
\noindent In order to simplify the technicality of this proof, $C$  is a constant of the form

\begin{equation*}
C =  C\left(\frac{1}{h_{\min}}, \mu_{0}, \lver \partial_{x}^{-2} k^{0} \rver_{H^{12}}, \lver \partial_{x}^{-1} v^{0} \rver_{H^{9}} \right)
\end{equation*}

\noindent The first and second point follow from Proposition \ref{existence_boussi} and Proposition \ref{existence_RKP}. Then, from System \eqref{wave_RKP} and Lemma \ref{control_diff_speed}, we obtain

\begin{equation*}
\lver \zeta_{(1)} \rver_{H^{2}} + \lver u_{(1)} \rver_{H^{2}} \leq C \frac{t}{1+t}.
\end{equation*} 

\noindent We also notice that we can control all the derivatives with respect to $x$, $y$ or $\tau$ of $u$ and $v$ be differentiating \eqref{wave_RKP}. Hence, we get a control for the remainders $R_{1}$ and $R_{2}$  and we obtain, for $0 \leq t \leq \frac{T}{\mu}$,

\begin{equation*}
\lver R_{1}(t) \rver_{H^{3}} + \lver R_{2}(t) \rver_{H^{3}} \leq C.
\end{equation*}
 
\noindent Then, using Proposition \ref{stability_boussi}, on can have

\begin{equation*}
\lver \left(\zeta_{B}, u_{B}, v_{B} \right) - \left(\zeta_{app}, u_{app}, v_{app} \right) \rver_{L^{\infty}([0,t] \times \RR)} \leq C \mu^{\frac{3}{2}} t.
\end{equation*}

\noindent Finally, from the ansatz \eqref{ansatz_RKP} and Lemma \ref{control_diff_speed}, we have

\begin{equation*}
\lver \left(\zeta_{app}, u_{app} \right) - \left(\zeta_{RKP},v_{RKP} \right) \rver_{L^{\infty}([0,t] \times \RR)} \leq \mu \frac{t}{1+t},
\end{equation*}

\noindent and the result follows easily.

\end{proof}

\noindent This theorem provides the first mathematical justification of the rotation-modified KP equation. Notice that the condition $k^{0} \in \partial_{x}^{2} H^{10}(\RR)$ is quite restrictive. As noted in \cite{Lannes_ww} Part 7.2.1 and in \cite{my_long_wave_corio} for the Ostrovsky equation, using the strategy developed in \cite{benyoussef_lannes}, we can hope to weaken the assumption on $k^{0}$ into $k^{0} \in \partial_{x} H^{9}(\RR)$.

\subsection{Very weak rotation, the KP equation}\label{model_KP}

\noindent In this part we study the situation of a very weak Coriolis forcing. We derive and fully justify the KP equation. We show that if $\frac{\epsilon}{\text{Ro}}$ is small enough, we can derive the KP equation

\begin{equation}\label{KP_eq}
\partial_{\xi} \left( \partial_{\tau} k +  \frac{3}{2} k \partial_{\xi} k + \frac{1}{6} \partial_{\xi}^{3} k \right) + \frac{1}{2} \partial_{yy} k = 0.
\end{equation}

\noindent Inspired by \cite{grimshaw_KP}, we consider the following asymptotic regime

\begin{equation*}
\mathcal{A}_{\text{KP}} = \left\{ \left(\mu, \epsilon, \gamma, \text{Ro}  \right), 0 \leq \mu \leq \mu_{0}, \epsilon = \mu, \gamma = \sqrt{\mu}, \frac{\epsilon}{\text{Ro}} = \mu \right\}.
\end{equation*}

\noindent The Boussinesq-Coriolis equations become $(\gamma = \sqrt{\mu})$

\begin{equation}\label{boussi_KP}
\left\{
\begin{aligned}
&\partial_{t} \zeta + \nabla^{\gamma} \cdot \left( [1+\mu \zeta] \overline{\textbf{V}} \right) = 0,\\
&\left(1- \frac{\mu}{3} \nabla^{\gamma} \nabla^{\gamma} \cdot \right) \partial_{t} \overline{\textbf{V}} + \nabla^{\gamma} \zeta + \mu \overline{\textbf{V}} \cdot \nabla^{\gamma} \overline{\textbf{V}} + \mu \overline{\textbf{V}}^{\perp} = 0.
\end{aligned}
\right.
\end{equation}

\noindent Proceeding as in the previous part, we denote $\textbf{V} = (u,v)^{t}$ and we seek an approximate solution $\left(\zeta_{app}, u_{app}, v_{app} \right)$ of \eqref{boussi_KP} in the form

\begin{equation}\label{ansatz_KP}
\begin{aligned}
&\zeta_{app}(t,x) = k(x-t,y,\mu t) + \mu \zeta_{(1)}(t,x,y,\mu t),\\
&u_{app}(t,x) = k(x-t,y,\mu t) + \mu u_{(1)}(t,x,y,\mu t),\\
&v_{app}(t,x) = \sqrt{\mu} v_{(1/2)}(t,x,y,\mu t) + \mu v_{(1)}(t,x,y,\mu t).
\end{aligned}
\end{equation}

\noindent Then, we plug the ansatz into Sytem \eqref{boussi_KP} and we get

\small
\begin{equation*}
\left\{
\begin{aligned}
&\partial_{t} \zeta_{app} + \nabla^{\gamma} \cdot \left( [1+\mu \zeta_{app}] \overline{\textbf{V}}_{app} \right) = \mu R^{1}_{(1)} + \mu^{\frac{3}{2}} \partial_{y} v_{(1)} + \mu^{2} R_{1},\\
&\left(1- \frac{\mu}{3} \nabla^{\gamma} \nabla^{\gamma} \cdot \right) \partial_{t} \overline{\textbf{V}}_{app} + \nabla^{\gamma} \zeta_{app} + \mu \overline{\textbf{V}}_{app} \cdot \nabla^{\gamma} \overline{\textbf{V}}_{app} + \mu \overline{\textbf{V}}_{app}^{\perp} = \sqrt{\mu} R^{2}_{(1/2)} + \mu R^{2}_{(1)} + \mu^{\frac{3}{2}} R_{2}.
\end{aligned}
\right.
\end{equation*}
\normalsize

\noindent where

\small
\begin{equation*}
\begin{aligned}
&R_{(1)}^{1} = \partial_{t} \zeta_{(1)} + \partial_{x} u_{(1)} + \partial_{\tau} k + 2 k \partial_{\xi} k  + \partial_{y} v_{(1/2)},\\
&R_{(1/2)}^{2} = \begin{pmatrix} 0 \\ \partial_{t} v_{(1/2)} + \partial_{y} k \end{pmatrix} \text{   and   } R_{(1)}^{2} = \begin{pmatrix} \partial_{t} u_{(1)} + \partial_{x} \zeta_{(1)} + \partial_{\tau} k +  \frac{1}{3} \partial^{3}_{\xi}  k + k \partial_{\xi} k \\ \partial_{t} v_{(1)} + k \end{pmatrix},\\
\end{aligned}
\end{equation*}
\normalsize

\noindent and 

\small
\begin{equation*}
\begin{aligned}
&R_{1} = \partial_{\tau} \zeta_{(1)}  + \partial_{x} \left(k u_{(1)} + k \zeta_{(1)} + \mu \zeta_{(1)} u_{(1)} \right) + \partial_{y} ((k + \mu \zeta_{(1)}) (v_{(1/2)} + \mu v_{(1)}) ) ,\\
&R_{2} = \left(- (v_{(1/2)} + \sqrt{\mu} v_{(1)}) + \sqrt{\mu} \tilde{R}_{2,1}, R_{2,2} \right)
\end{aligned}
\end{equation*}
\normalsize

\noindent with

\small
\begin{align*}
&\tilde{R}_{2,1} =  \partial_{\tau} u_{(1)} - \frac{1}{3} \partial_{\xi}^{2} \partial_{\tau} k - \frac{1}{3} \partial_{x}^{2} \partial_{t} u_{(1)}  - \mu \frac{1}{3} \partial_{x}^{2} \partial_{\tau} u_{(1)} + \partial_{x} \left(k u_{(1)} \right) + \mu u_{(1)} \partial_{x} u_{(1)}\\
&\hspace{1cm}  - \frac{1}{3} \partial^{3}_{xyt} (v_{(1/2)} + \sqrt{\mu} v_{(1)}) - \frac{\mu}{3} \partial^{3}_{xy \tau} (v_{(1/2)} + \sqrt{\mu} v_{(1)}) + (v_{(1/2)} + \sqrt{\mu} v_{(1)}) \partial_{y} \left(k + \mu u_{(1)} \right),\\
&R_{2,2} = \partial_{\tau} v_{(1/2)} + \partial_{y} \zeta_{(1)} + k \partial_{x} (v_{(1/2)} + \sqrt{\mu} v_{(1)}) + u_{(1)} + \frac{1}{3} \partial_{y} \partial_{\xi}^{2} k + \mu u_{(1)} \partial_{x} (v_{(1/2)} + \sqrt{\mu} v_{(1)})\\
&\hspace{1cm} - \frac{\mu}{3} \left( \partial^{3}_{y x \tau} k + \partial^{3}_{y x t} u_{(1)} + \partial_{y}^{2} \partial_{t} (v_{(1/2)} + \sqrt{\mu} v_{(1)}) + \mu \partial_{y x \tau} u_{(1)} + \mu \partial_{y}^{2} \partial_{\tau} (v_{(1/2)} + \sqrt{\mu} v_{(1)}) \right)\\
&\hspace{1cm} + \mu (v_{(1/2)} + \sqrt{\mu} v_{(1)}) \partial_{y} (v_{(1/2)} + \sqrt{\mu} v_{(1)}).
\end{align*}
\normalsize

\noindent Then, we choose $(k,v_{(1/2)},v_{(1)})$ such that, for all $(x,y) \in \RR$, $t \in \left[0, \frac{T}{\mu} \right]$ and $\tau \in \left[0,T \right]$, 

\begin{equation*}
R_{(1)}^{1}(t,x,y,\tau) = 0 \text{   and   } R_{(1/2)}^{2}(t,x,y,\tau) = R_{(1)}^{2}(t,x,y,\tau) = 0.
\end{equation*}

\noindent First, we obtain that 

\begin{align*}
&v_{(1/2)} = \partial_{x}^{-1} \partial_{y} k + v_{(1/2)}^{0} - \partial_{x}^{-1} \partial_{y} k^{0},\\
&v_{(1)} = \partial_{x}^{-1} k + v_{(1)}^{0} - \partial_{x}^{-1} k^{0}.
\end{align*}

\noindent Then, denoting $w_{\pm} = \zeta_{(1)} \pm u_{(1)}$, we get

\small
\begin{equation*}
\begin{aligned}
&\left(\partial_{t} + \partial_{x} \right) w_{+} + \left( 2 \partial_{\tau} k + 3 k \partial_{\xi} k + \frac{1}{3} \partial_{\xi}^{3} k + \partial_{\xi}^{-1} \partial_{y}^{2} k \right) (x-t,\tau) + F_{0} = 0,\\
&\left(\partial_{t} - \partial_{x} \right) w_{-} + \left(k \partial_{\xi} k - \frac{1}{3} \partial_{\xi}^{3} k + \partial_{\xi}^{-1} \partial_{y}^{2} k \right) (x-t,\tau) + F_{0} = 0,\\
\end{aligned}
\end{equation*}
\normalsize

\noindent where 

\begin{equation*}
F_{0} =  \partial_{y} v_{(1/2)}^{0} -  \partial_{\xi}^{-1} \partial_{y}^{2} k^{0}.
\end{equation*}

\noindent Therefore, in order to avoid a linear growth (see Lemma \ref{control_diff_speed}), $k$ must satisfies the KP equation \eqref{KP_eq}. The following Lemma is a local wellposedness result for the KP equation (see Lemma 7.22 in \cite{Lannes_ww} or \cite{Saut_KP,bourgainKP,ukaiKP}). 

\begin{prop}\label{existence_KP}
\noindent Let $N \geq 5$ and $k_{0} \in \partial_{x} H^{N}(\RR)$. Then, there exists a time $T > 0$ and a unique solution $k \in \mathcal{C} \left([0,T]; \partial_{x} H^{N}(\RR) \right)$ to the KP equation \eqref{KP_eq} and one has

\begin{equation*}
\lver \partial_{\xi}^{-1} k(t,\cdot) \rver_{H^{N}} \leq C \left(T, \lver \partial_{\xi}^{-1} k_{0} \rver_{H^{N}} \right).
\end{equation*}

\noindent Furthermore, if $N \geq 6$ and $\partial_{y}^{2} k_{0} \in \partial_{x}^{2} H^{N-4}(\RR)$, then $\partial_{y}^{2} k \in \mathcal{C} \left([0,T]; \partial_{x}^{2} H^{N-4}(\RR) \right)$ and one has

\begin{align*}
&\lver \partial_{y}^{2} k_{0} \partial_{\xi}^{-2} k(t,\cdot) \rver_{H^{N-4}} \leq C \left(T, \lver \partial_{\xi}^{-2} \partial_{y}^{2} k_{0} \rver_{H^{N-4}}, \lver \partial_{\xi}^{-1} k_{0} \rver_{H^{N}} \right).
\end{align*}
\end{prop}

\noindent We can now establish a rigorous justification of the KP equation.

\begin{thm}\label{full_just_KP}
\noindent Let $k^{0} \in \partial_{x}^{2} H^{12}(\RR)$ such that $1+\epsilon k^{0} \geq h_{\min} > 0$ and $v_{(1/2)}^{0} \in \partial_{x} H^{8}(\RR)$, $v_{(1)}^{0} \in H^{7}(\RR)$. Suppose that \upshape$\left( \mu,\epsilon,\gamma,\text{Ro} \right) \in \mathcal{A}_{\text{KP}}$\itshape. Denote $v^{0} = \sqrt{\mu} v^{0}_{(1/2)} + \mu v^{0}_{(1)}$. Then, there exists a time $T_{0} > 0$, such that we have

\medskip

\noindent (i) a unique classical solution \upshape$\left(\zeta_{B},u_{B}, v_{B} \right)$ \itshape of \eqref{boussi_weakrot} with initial data \upshape$\left(k^{0}, k^{0}, v^{0} \right)$ \itshape on $\left[ 0, \frac{T_{0}}{\mu} \right]$.

\medskip

\noindent (ii)  a unique classical solution \upshape$k$ \itshape of \eqref{KP_eq} with initial data \upshape$k^{0}$ \itshape on $\left[ 0, T_{0} \right]$.

\medskip 

\noindent (iii) If we define $\left(\zeta_{KP}, u_{KP} \right)(t,x) = \left( k(x-t,y,\mu t),k(x-t,y,\mu t ) \right)$ we have the following error estimate for all $0 \leq t \leq \frac{T_{0}}{\mu}$,

\begin{equation*}
\lver \left(\zeta_{B},u_{B} \right) - \left(\zeta_{KP},u_{KP} \right) \rver_{L^{\infty}([0,t] \times \RR)} \leq  C \frac{\mu t}{1+t} (1+ \sqrt{\mu} t) 
\end{equation*}
\itshape

\noindent where $C = C\left(\frac{1}{h_{\min}}, \mu_{0}, \lver \partial_{x}^{-2} k^{0} \rver_{H^{12}}, \lver \partial_{x}^{-1} v^{0}_{(1/2)} \rver_{H^{8}}, \lver v^{0}_{(1)} \rver_{H^{7}} \right)$.
\end{thm}

\begin{proof}
\noindent The proof is very similar to the proof of Theorem \ref{full_just_RKP}.
\end{proof}

\begin{remark}
\noindent Contrary to the justification of the KP equation in the irrotational setting (see Part 7.2 in \cite{Lannes_ww} or \cite{Saut_Lannes_KP}), the transverse part of the horizontal velocity $v$ must contain an order $\mathcal{O} (\mu)$ contribution. Notice that if one considers a weaker Coriolis forcing, for instance $\frac{\epsilon}{\text{Ro}} = \mu^{\frac{3}{2}}$, this assumption is no more necessary.
\end{remark}

\section{Which equation for which asymptotic regime ?}\label{comparaison_models}

\subsection{The Ostrovsky and KdV equations}\label{ostrov_kdv}

\noindent In Section \ref{KP_approx}, we derived two asymptotic models in the long wave regime $(\epsilon=\mu)$. First, if $\gamma=\sqrt{\mu}$ and $\frac{\epsilon}{\text{Ro}} = \sqrt{\mu}$, we derived the rotation-modified KP equation

\begin{equation*}
\partial_{\xi} \left( \partial_{\tau} k +  \frac{3}{2} k \partial_{\xi} k + \frac{1}{6} \partial_{\xi}^{3} k \right) + \frac{1}{2} \partial_{yy} k = \frac{1}{2} k.
\end{equation*}

\noindent Then, if $\gamma=\sqrt{\mu}$ and $\frac{\epsilon}{\text{Ro}} = \mu$, we obtained the KP equation

\begin{equation*}
\partial_{\xi} \left( \partial_{\tau} k +  \frac{3}{2} k \partial_{\xi} k + \frac{1}{6} \partial_{\xi}^{3} k \right) + \frac{1}{2} \partial_{yy} k = 0.
\end{equation*}

\noindent In \cite{my_long_wave_corio}, we performed a similar derivation in the long wave regime under the assumption that $\gamma = \mathcal{O} (\mu^{2})$. When $\frac{\epsilon}{\text{Ro}} = \sqrt{\mu}$, we derived the Ostrovsky equation

\begin{equation}\label{ostrov_eq}
\partial_{\xi} \left( \partial_{\tau} k +  \frac{3}{2} k \partial_{\xi} k + \frac{1}{6} \partial_{\xi}^{3} k \right) = \frac{1}{2} k,
\end{equation}

\noindent and when $\frac{\epsilon}{\text{Ro}} = \mu$, we derived the KdV equation

\begin{equation}\label{kdv_eq}
\partial_{\tau} k +  \frac{3}{2} k \partial_{\xi} k + \frac{1}{6} \partial_{\xi}^{3} k = 0.
\end{equation}

\noindent We would like to emphasize that we can weaken the assumption $\gamma = \mathcal{O} (\mu^{2})$ into $\gamma = \mu$. In the following, we show this fact on the Ostrovsky equation. We consider the asymptotic regime

\begin{equation*}
\mathcal{A}_{\text{ostrov}} = \left\{ \left(\mu, \epsilon, \gamma, \text{Ro} \right), 0 \leq \mu \leq \mu_{0}, \epsilon = \mu, \gamma = \mu, \frac{\epsilon}{\text{Ro}} = \sqrt{\mu} \right\}.
\end{equation*}

\noindent Then we seek an approximate solution $\left(\zeta_{app}, u_{app}, v_{app} \right)$ of the Boussinesq-Coriolis equations in the form

\begin{equation*}
\begin{aligned}
&\zeta_{app}(t,x,y) = k(x-t,y,\mu t) + \mu \zeta_{(1)}(t,x,y,\mu t),\\
&u_{app}(t,x,y) = k(x-t,y,\mu t) + \mu u_{(1)}(t,x,y,\mu t),\\
&v_{app}(t,x,y) = \sqrt{\mu} v_{(1/2)}(t,x,y,\mu t) + \mu v_{(1)}(t,x,y,\mu t)
\end{aligned}
\end{equation*}

\noindent Plugging the ansatz into the Boussinesq-Coriolis equations, we obtain

\small
\begin{equation*}
\left\{
\begin{aligned}
&\partial_{t} \zeta_{app} + \nabla^{\gamma} \cdot \left( [1+\mu \zeta_{app}] \overline{\textbf{V}}_{app} \right) = \mu R^{1}_{(1)} + \mu^{\frac{3}{2}} R_{1},\\
&\left(1- \frac{\mu}{3} \nabla^{\gamma} \nabla^{\gamma} \cdot \right) \partial_{t} \overline{\textbf{V}}_{app} + \nabla^{\gamma} \zeta_{app} + \mu \overline{\textbf{V}}_{app} \cdot \nabla^{\gamma} \overline{\textbf{V}}_{app} + \sqrt{\mu} \overline{\textbf{V}}_{app}^{\perp} = \sqrt{\mu} R^{2}_{(1/2)} + \mu R^{2}_{(1)} + \mu^{\frac{3}{2}} R_{2}.
\end{aligned}
\right.
\end{equation*}
\normalsize

\noindent where

\small
\begin{equation*}
\begin{aligned}
&R_{(1)}^{1} = \partial_{t} \zeta_{(1)} + \partial_{x} u_{(1)} + \partial_{\tau} k + 2 k \partial_{\xi} k,\\
&R_{(1/2)}^{2} = \begin{pmatrix} 0 \\ \partial_{t} v_{(1/2)} + k \end{pmatrix} \text{   and   } R_{(1)}^{2} = \begin{pmatrix} \partial_{t} u_{(1)} + \partial_{x} \zeta_{(1)} + \partial_{\tau} k +  \frac{1}{3} \partial^{3}_{\xi}  k + k \partial_{\xi} k - v_{(1/2)} \\ \partial_{t} v_{(1)} + \partial_{y} k \end{pmatrix},\\
\end{aligned}
\end{equation*}
\normalsize

\noindent and where $R_{1}$, $R_{2}$ are remainders similar to the ones found in Sections \ref{model_RKP} and \ref{model_KP}. Then, using the same strategy than before, we impose that $R^{1}_{(1)} = 0$ and $R^{2}_{(1/2)} = R^{2}_{(1)} = 0$. We obtain

\begin{align*}
&v_{(1/2)} = \partial_{\xi}^{-1} k + v_{(1/2)}^{0} - \partial_{\xi}^{-1} k^{0},\\
&v_{(1)} = \partial_{\xi}^{-1} \partial_{y} k + v_{(1)}^{0} - \partial_{\xi}^{-1} \partial_{y} k^{0},
\end{align*}

\noindent and, denoting $w_{\pm} = \zeta_{(1)} \pm u_{(1)}$, we get

\small
\begin{equation*}
\begin{aligned}
&\left(\partial_{t} + \partial_{x} \right) w_{+} + \left( 2 \partial_{\tau} k + 3 k \partial_{\xi} k + \frac{1}{3} \partial_{\xi}^{3} k - \partial_{\xi}^{-1} k \right) (x-t,\tau) - F_{0} = 0,\\
&\left(\partial_{t} - \partial_{x} \right) w_{-} + \left(k \partial_{\xi} k - \frac{1}{3} \partial_{\xi}^{3} k + \partial_{\xi}^{-1} k \right) (x-t,\tau) + F_{0} = 0,\\
\end{aligned}
\end{equation*}
\normalsize

\noindent where $F_{0} =  v_{(1/2)}^{0} - \partial_{\xi}^{-1} k^{0}$. In order to avoid a linear growth (see Lemma \ref{control_diff_speed}), $k$ must satisfies the Ostrovsky equation \eqref{ostrov_eq}. Proceeding as in \cite{my_long_wave_corio}, we can generalize Theorem 3.9 in \cite{my_long_wave_corio} to the asymptotic regime $\mathcal{A}_{\text{ostrov}}$. A solution of the Ostrovsky equation provides a $\mathcal{O}(\sqrt{\mu})$ approximation of the Boussinesq-Coriolis equations over a time $\mathcal{O} \left( \frac{1}{\mu} \right)$. We can proceed similarly for the KdV equation \eqref{kdv_eq}. Under the asymptotic regime

\begin{equation*}
\mathcal{A}_{\text{KdV}} = \left\{ \left(\mu, \epsilon, \gamma, \text{Ro} \right), 0 \leq \mu \leq \mu_{0}, \epsilon = \mu, \gamma = \mu, \frac{\epsilon}{\text{Ro}} = \mu \right\}.
\end{equation*}

\noindent and with the ansatz

\begin{equation*}
\begin{aligned}
&\zeta_{app}(t,x,y) = k(x-t,y,\mu t) + \mu \zeta_{(1)}(t,x,y,\mu t),\\
&u_{app}(t,x,y) = k(x-t,y,\mu t) + \mu u_{(1)}(t,x,y,\mu t),\\
&v_{app}(t,x,y) = \mu v_{(1)}(t,x,y,\mu t),
\end{aligned}
\end{equation*}

\noindent we can generalize Theorem 3.12 in \cite{my_long_wave_corio} to the asymptotic regime $\mathcal{A}_{\text{KdV}}$. A solution of the KdV equation provides a $\mathcal{O}(\mu)$ approximation of the Boussinesq-Coriolis equations over a time $\mathcal{O} \left( \frac{1}{\mu} \right)$.

\subsection{Conclusion}

\noindent We summarize Section \ref{KP_approx} and Subsection \ref{ostrov_kdv} by the following table. Notice that all of these models provide a $\mathcal{O}(\sqrt{\mu})$ approximation (at least) in the long wave regime ($\epsilon=\mu$) of the Boussinesq-Coriolis equations over a time $\mathcal{O} \left( \frac{1}{\mu} \right)$. 

\medskip
\medskip

\begin{table}[!h]
\centering
\begin{tabular}{|c|c|c|}
  \hline
  \diagbox[width=3.5em]{$\gamma$}{$\frac{\epsilon}{\text{Ro}}$} & $\sqrt{\mu}$ & $\mu$ \\\hline
  $\sqrt{\mu}$ & Rotation-modified KP equation & KP equation \\\hline
  $\;\; \mu $ & Ostrovsky equation & KdV equation \\\hline
\end{tabular}
\end{table}

\section*{Acknowledgments}

\noindent The author would like to thank Jean-Claude Saut for the fruitful discussions about the KP approximation.

\newpage
\footnotesize
\bibliographystyle{plain}
\bibliography{biblio}
\normalsize

\end{document}